\documentclass[10pt]{amsart}
\usepackage{amsmath, amsthm, amsfonts, amssymb} 
\usepackage{url}
\usepackage{mathrsfs}  
\renewcommand{\mathcal}{\mathscr}

\newtheorem{theo}{Theorem}[section]

\newtheorem*{theoremA}{Theorem A}
\newtheorem*{theoremB}{Theorem B}

\newtheorem*{corollaryC}{Corollary C}
\newtheorem*{theoremD}{Theorem D}
\newtheorem*{corollaryE}{Corollary E}

\newtheorem{cor}[theo]{Corollary}
\newtheorem{lem}[theo]{Lemma}
\newtheorem{prop}[theo]{Proposition}
\newtheorem{claim}[theo]{Claim}
\theoremstyle{definition}
\newtheorem{rem}[theo]{Remark}

\newtheorem{df}[theo]{Definition}

\newcommand{\R}{\mathbf{R}}
\newcommand{\C}{\mathbf{C}}
\newcommand{\Z}{\mathbf{Z}}
\newcommand{\F}{\mathbf{F}}

\newcommand{\N}{\mathbf{N}}

\newcommand{\m}{\operatorname{m}}
\newcommand{\Sd}{\operatorname{Sd}}

\newcommand{\pol}{\operatorname{pol}}
\newcommand{\Id}{\operatorname{Id}}
\newcommand{\Ad}{\operatorname{Ad}}

\newcommand{\Tr}{\operatorname{Tr}}

\newcommand{\Aut}{\operatorname{Aut}}
\newcommand{\SL}{\operatorname{SL}}

\newcommand{\dom}{\operatorname{dom}}

\newcommand{\op}{\operatorname{op}}
\newcommand{\co}{\operatorname{co}}

\newcommand{\Inn}{\operatorname{Inn}}
\newcommand{\Out}{\operatorname{Out}}

\newcommand{\cb}{\operatorname{cb}}

\newcommand{\QN}{\mathcal{Q}\mathcal{N}}

\newcommand{\Diag}{\operatorname{Diag}}
\newcommand{\Ch}{\mathcal{T}(H)}
\newcommand{\GG}{\Gamma(H_\R, U_t)}
\newcommand{\G}{\Gamma(H_\R, U_t)''}
\newcommand{\vect}{\operatorname{span}}
\newcommand{\weak}{\operatorname{weak}}

\begin{document}

\title[Free Araki-Woods factors]{Approximation properties and absence of Cartan subalgebra for free Araki-Woods factors}

\begin{abstract}
We show that all the free Araki-Woods factors $\Gamma(H_\R, U_t)''$ have the complete metric approximation property. Using Ozawa-Popa's techniques, we then prove that every nonamenable subfactor $\mathcal{N} \subset \Gamma(H_\R, U_t)''$ which is the range of a normal conditional expectation has no Cartan subalgebra. We finally deduce that the type ${\rm III_1}$ factors constructed by Connes in the '70s can never be isomorphic to any free Araki-Woods factor, which answers a question of Shlyakhtenko and Vaes.
\end{abstract}

\author[C. Houdayer]{Cyril Houdayer*}

\address{CNRS-ENS Lyon \\
UMPA UMR 5669 \\
69364 Lyon cedex 7 \\
France}

\thanks{*Research partially supported by ANR grant Agora NT09-461407}
\thanks{**Research partially supported by ANR grant 06-BLAN-0015}

\email{cyril.houdayer@ens-lyon.fr}

\author[\'E. Ricard]{\'Eric Ricard**}

\address{Laboratoire de math\'ematiques \\ 
Universit\'e de Franche-Comt\'e \\
 16 route de Gray \\
25030 Besan\c{c}on \\
 France}
 
 \email{eric.ricard@univ-fcomte.fr}

\subjclass[2000]{46L07; 46L10; 46L54}

\keywords{Cartan subalgebras; Type ${\rm III}$ factors; Complete metric approximation property; Deformation/rigidity; Free probability}

\maketitle

\section{Introduction and statement of the main results}

The {\it free Araki-Woods factors} were introduced by Shlyakhtenko \cite{shlya97}. In the context of free probability theory, these factors can be regarded as analogs of the hyperfinite factors coming from the CAR\footnote{Canonical Anticommutation Relations} functor. To each real separable Hilbert space $H_\R$ together with an orthogonal representation $(U_t)$ of $\R$ on $H_\R$, one  associates \cite{shlya97} a von Neumann algebra denoted by $\Gamma(H_\R, U_t)''$, called the {\it free Araki-Woods} von Neumann algebra. The von Neumann algebra $\Gamma(H_\R, U_t)''$ comes equipped with a unique {\it free quasi-free state}, which is always normal and faithful (see Section \ref{preliminaries} for a more detailed construction). If $\dim H_\R = 1$, then $\Gamma(\R, \Id)'' \cong L^\infty([0, 1])$. If $\dim H_\R \geq 2$, then $\mathcal{M} = \Gamma(H_\R, U_t)''$ is a full factor. In particular, $\mathcal{M}$ can never be of type ${\rm III_0}$. The type classification of these factors is the following:
\begin{enumerate}
\item $\mathcal{M}$ is a type ${\rm II_1}$ factor if and only if the representation $(U_t)$ is trivial: in that case the functor $\Gamma$ is Voiculescu's free Gaussian functor \cite{voiculescu92}. Then $\Gamma(H_\R, 1)'' \cong L(\F_{\dim H_\R})$ is a free group factor.
\item $\mathcal{M}$ is a type ${\rm III_\lambda}$ factor, for $0 < \lambda < 1$, if and only if the representation $(U_t)$ is $\frac{2\pi}{|\log \lambda|}$-periodic. 
\item $\mathcal{M}$ is a type ${\rm III_1}$ factor if and only if $(U_t)$ is nonperiodic and nontrivial.
\end{enumerate}

Let us start by recalling some fundamental structural results for free group factors. In their breakthrough paper \cite{ozawapopa}, Ozawa and Popa showed that the free group factors $L(\F_n)$ are {\em strongly solid}, i.e. the normalizer $\mathcal{N}_{L(\F_n)}(P)=\{u\in \mathcal{U}(L(\F_n)): uPu^*=P\}$ of any diffuse amenable subalgebra $P\subset L(\F_n)$ generates an amenable von Neumann algebra, thus hyperfinite by Connes' result \cite{connes76}. This strengthened two well-known indecomposability results for free group factors: Voiculescu's celebrated result in \cite{voiculescu96}, showing that $L(\F_n)$ has no Cartan subalgebra, which in fact exhibited the first examples of factors with no Cartan decomposition; and Ozawa's result in \cite{ozawa2003}, showing that the commutant in $L(\F_n)$ of any diffuse subalgebra must be amenable ($L(\F_n)$ are {\it solid}).

For the type ${\rm III}$ free Araki-Woods factors $\mathcal{M} = \Gamma(H_\R, U_t)''$,  Shlyakhtenko obtained several remarkable classification results using free probability techniques:
\begin{itemize}
\item When $(U_t)$ are {\it almost periodic}, the free Araki-Woods factors are completely classified up to state-preserving $\ast$-isomorphism \cite{shlya97}: they only depend on Connes' invariant $\Sd(\mathcal{M})$ which is equal in that case to the (countable) subgroup $S_U \subset \R_+$ generated by the eigenvalues of $(U_t)$. Moreover, the {\it discrete} core $\mathcal{M} \rtimes_\sigma \widehat{S_U}$ (where $\widehat{S_U}$ is the Pontryagin dual of $S_U$) is $\ast$-isomorphic to $L(\F_\infty) \overline{\otimes} \mathbf{B}(\ell^2)$.
\item If $(U_t)$ is the left regular representation, then the {\it continuous} core $M = \mathcal{M} \rtimes_\sigma \R$ is $\ast$-isomorphic to $L(\F_\infty) \overline{\otimes} \mathbf{B}(\ell^2)$ \cite{shlya98} and the dual ``trace-scaling" action $(\theta_s)$ is precisely the one constructed by  R\u{a}dulescu \cite{radulescu1991}.
\end{itemize}
For more on free Araki-Woods factors, we refer to \cite{{houdayer6}, {houdayer5}, {houdayer3}, {houdayer2}, {shlya2004}, {shlya2003}, {shlya2000}, {shlya99}, {shlya98}, {shlya97}} and also to Vaes' Bourbaki seminar \cite{vaes2004}.

Our first result deals with approximation properties for $\Gamma(H_\R, U_t)''$. Recall that a von Neumann algebra $\mathcal{N}$ is said to have the {\it complete metric approximation property} (c.m.a.p.) \cite{haa} if there exists a net of normal finite rank completely bounded maps $\Phi_n : \mathcal{N} \to \mathcal{N}$ such that
\begin{itemize}
\item $\Phi_n(x) \to x$ $\ast$-strongly, for every $x \in \mathcal{N}$;
\item $\|\Phi_n\|_{\cb} \leq 1$, for every $n$.
\end{itemize}

Haagerup first established in \cite{haagerup84} that the free group
factors $L(\F_n)$ have the metric approximation property. His idea
was to use radial multipliers on $\F_n$. In a subsequent
unpublished work with Szwarc (see \cite{HSS}), a complete description
of completely bounded radial multipliers was obtained, showing
that $L(\F_n)$ has the complete metric approximation property.  Along
the same pattern, we start by characterizing appropriate {\it radial
multipliers} on $\Gamma(H_\R, U_t)''$. At the $L^2$-level, that is on
the Fock space, they just act diagonally on tensor powers of $H$.
They allow us to reduce the question of the approximation property to a
finite length situation, which is enough to conclude for almost
periodic representations $(U_t)$. To proceed to the general case, we use completely
positive maps arising from the second quantization functor.  The
novelty here is that it holds true under a milder assumption than the
usual one \cite{{shlya97}, {voiculescu85}}, and we obtain:

\begin{theoremA}
All the free Araki-Woods factors have the complete metric approximation property.
\end{theoremA}

The free Araki-Woods factors $\Gamma(H_\R, U_t)''$ as well as their continuous cores carry a {\em free malleable deformation} $(\alpha_t)$ in the sense of Popa: it naturally arises from the second quantization of the rotations defined on $H_\R \oplus H_\R$ that commute with $U_t \oplus U_t$. Using Ozawa-Popa's techniques \cite{{ozawapopa}, {ozawapopaII}}, we will then apply the {\it deformation/rigidity} strategy together with the {\it intertwining techniques} in order to study $\Gamma(H_\R, U_t)''$. The high flexibility of this approach will allow us to work in a {\it semifinite} setting, so that we can obtain new structural/indecomposability results for the free Araki-Woods factors as well as their continuous cores. Recall in that respect that a von Neumann subalgebra $A \subset \mathcal{M}$ is said to be a {\em Cartan subalgebra} if the following conditions hold:
\begin{itemize}
\item $A$ is maximal abelian, i.e. $A = A' \cap \mathcal{M}$.
\item There exists a faithful normal conditional expectation $E : \mathcal{M} \to A$.
\item The normalizer $\mathcal{N}_{\mathcal M}(A) = \{u \in \mathcal{U}(\mathcal{M}) : u A u^* = A\}$ generates $\mathcal{M}$.
\end{itemize}
It follows from \cite{FM} that in that case, $L^\infty(X, \mu) = A \subset \mathcal{M} = L(\mathcal{R}, \omega)$ is the von Neumann algebra of a nonsingular equivalence relation $\mathcal{R}$ on the standard probability space $(X, \mu)$ up to a scalar $2$-cocycle $\omega$ for $\mathcal{R}$.

Shlyakhtenko showed \cite{shlya2000} that the unique type ${\rm III_\lambda}$ free Araki-Woods factor ($0 < \lambda < 1$) has no Cartan subalgebra. We generalize this result and prove the analog of the strong solidity \cite{ozawapopa} for all the free Araki-Woods factors. Our second result is the following global dichotomy result for conditioned diffuse subalgebras of free Araki-Woods factors.

\begin{theoremB}
Let $\mathcal{M} = \Gamma(H_\R, U_t)''$ be any free Araki-Woods factor. Let $\mathcal{N} \subset \mathcal{M}$ be a diffuse von Neumann subalgebra for which there exists a faithful normal conditional expectation $E : \mathcal{M} \to \mathcal{N}$. Then either $\mathcal{N}$ is hyperfinite or $\mathcal{N}$ has no Cartan subalgebra.
\end{theoremB}

We can deduce from Theorems A and B new classification results for the free Araki-Woods factors. First recall that a factor $\mathcal{N}$ is said to be {\em full} if the subgroup of inner automorphisms $\Inn(\mathcal{N})$ is closed in $\Aut(\mathcal{N})$. Write $\pi : \Aut(\mathcal{N}) \to \Out(\mathcal{N})$ for the quotient map. For a full type ${\rm III_1}$ factor $\mathcal{N}$, Connes' invariant $\tau(\mathcal{N})$ is defined as the weakest topology on $\R$ that makes the map $t \mapsto \pi(\sigma_t^\varphi) \in \Out(\mathcal{N})$ continuous. In \cite{connes74}, Connes constructed type ${\rm III_1}$ factors $\mathcal{N}$ with prescribed $\tau$ invariant. Recall his construction. Let $\mu$ be a finite Borel measure on $\R_+$ such that $\int \lambda \, {\rm d} \mu(\lambda) < \infty$. We will normalize $\mu$ so that $\int (1 + \lambda) \, {\rm d} \mu(\lambda) = 1$. Define the unitary representation $(U_t)$ of $\R$ on the real Hilbert space $L^2(\R_+, \mu)$ by $(U_t \xi)(\lambda) = \lambda^{it} \xi(\lambda)$. We will assume that $(U_t)$ is not periodic. Define on $P = \mathbf{M}_2(\C) \otimes L^\infty(\R_+, \mu)$ the faithful normal state $\varphi$ by
$$\varphi\begin{pmatrix}
f_{11} & f_{12} \\
f_{21} & f_{22}
\end{pmatrix} = \int f_{11}(\lambda) \, {\rm d} \mu(\lambda) + \int \lambda f_{22}(\lambda) \, {\rm d} \mu(\lambda).$$
Let $\F_n$ be acting by Bernoulli shift on 
$$\mathcal{P}_\infty = \overline{\bigotimes_{g \in \F_n}} (P, \varphi).$$
Denote by $\mathcal{N} = \mathcal{P}_\infty \rtimes \F_n$ the corresponding crossed product. By the general theory, $\mathcal{N}$ is a type ${\rm III_1}$ factor. Connes showed that $\mathcal{N}$ is a full factor and $\tau(\mathcal{N})$ is the weakest topology that makes the map $t \mapsto U_t$ $\ast$-strongly continuous. In particular, if $(U_t)$ is the left regular representation, then $\tau(\mathcal{N})$ is the usual topology and $\mathcal{N}$ has no almost periodic state. Observe that $\mathcal{N}$ has a Cartan subalgebra $A$ given by
$$A = \overline{\bigotimes_{g \in \F_n}} \Diag_2(L^\infty(\R_+, \mu)).$$
The following Corollary answers a question of Shlyakhtenko (see \cite[Problem 8.7]{shlya99}) and Vaes (see \cite[Remarque 2.8]{vaes2004}).

\begin{corollaryC}
The type ${\rm III_1}$ factors constructed by Connes are never isomorphic to any free Araki-Woods factor. More generally, they cannot be conditionally embedded into a free Araki-Woods factor.
\end{corollaryC}

The continuous cores $M = \Gamma(H_\R, U_t)'' \rtimes_\sigma \R$ of the free Araki-Woods factors were shown to be {\em semisolid}\footnote{For every finite projection $p \in M$, the commutant in $pMp$ of any type ${\rm II_1}$ subalgebra must be amenable.} for every orthogonal representation $(U_t)$ and {\em solid} when $(U_t)$ is strongly mixing (see \cite[Theorem 1.1]{houdayer6}). They moreover have the c.m.a.p. by Theorem A. Using a similar strategy as in \cite{houdayer8}, we obtain new structural results for the continuous cores of the free Araki-Woods factors. 

\begin{theoremD}
Let $(U_t)$ be a nontrivial nonperiodic orthogonal representation of $\R$. Denote by $\mathcal{M} = \Gamma(H_\R, U_t)''$ the corresponding type ${\rm III_1}$ free Araki-Woods factor. Denote by $M = \mathcal{M} \rtimes_\sigma \R$ its continuous core, which is a type ${\rm II_\infty}$ factor. Let $p \in M$ be a nonzero finite projection and write $N = pMp$.

\begin{enumerate}
\item For any maximal abelian $\ast$-subalgebra $A \subset N$, $\mathcal{N}_N(A)''$ is amenable. In particular, $N$ has no Cartan subalgebra.

\item Assume that
\begin{enumerate}
\item either $(U_t)$ is strongly mixing;
\item or $U_t = \R \oplus V_t$, where $(V_t)$ is strongly mixing.
\end{enumerate}
Then for any diffuse amenable von Neumann subalgebra $P \subset N$, $\mathcal{N}_N(P)''$ is amenable, i.e.\ $N$ is strongly solid.
\end{enumerate}
\end{theoremD}

The proof of Theorems B and D is a combination of ideas and techniques of \cite{{houdayer4}, {houdayer8}, {houdayer6}, {ozawapopa}, {ozawapopaII}} and rely on Theorem A. Note that Theorem D allows us to obtain other new classification results. Indeed let $\SL_n(\Z) \curvearrowright \R^n$ be the linear action. Observe that it is an infinite measure-preserving free ergodic action. Thus the corresponding crossed product von Neumann algebra $Q_n = L^\infty(\R^n) \rtimes \SL_n(\Z)$ is a ${\rm II_\infty}$ factor, which is nonamenable for $n \geq 3$. Since the dilation $d_t : \R^n \ni x \mapsto tx \in \R^n$ (for $t > 0$) commutes with $\SL_n(\Z)$, it gives a trace-scaling action $(\theta_t) : \R_+ \curvearrowright Q_n$. Theorem D implies in particular that the type ${\rm III_1}$ factors $Q_n \rtimes_{(\theta_t)} \R_+$ obtained this way cannot be isomorphic to any free Araki-Woods factor.

Using \cite[Theorem 2.5, ${\rm v}$]{antoniou} (see also the discussion in \cite[4.2]{houdayer8}), we can construct an example of an orthogonal representation $(U_t)$ of $\R$ on a (separable) real Hilbert space $H_\R$ such that:
\begin{enumerate}
\item $(U_t)$ is strongly mixing.
\item The spectral measure of $\bigoplus_{n \geq 1} U_t^{\otimes n}$ is singular with respect to the Lebesgue measure on $\R$.
\end{enumerate}
Shlyakhtenko showed \cite[Theorem 9.12]{shlya99} that if the spectral measure of the representation $\bigoplus_{n \geq 1} U_t^{\otimes n}$ is singular with respect to the Lebesgue measure, then the continuous core of the free Araki-Woods factor $\Gamma(H_\R, U_t)''$ cannot be isomorphic to any $L(\F_t) \overline{\otimes} \mathbf{B}(\ell^2)$, for $1 < t \leq \infty$, where $L(\F_t)$ denote the interpolated free group factors \cite{{dykema94}, {radulescu1994}}. Therefore, we obtain:

\begin{corollaryE}
Let $(U_t)$ be an orthogonal representation acting on $H_\R$ as above. Denote by $\mathcal{M} = \Gamma(H_\R, U_t)''$ the corresponding free Araki-Woods factor and by $M = \mathcal{M}Ê\rtimes_\sigma \R$ its continuous core. Let $p \in M$ be a nonzero finite projection and write $N = pMp$. We have
\begin{itemize}
\item $N$ is a nonamenable strongly solid ${\rm II_1}$ factor with the c.m.a.p.\ and the Haagerup property.
\item $N$ is not isomorphic to any interpolated free group factor $L(\F_t)$, for $1 < t \leq \infty$;
\item $N \overline{\otimes} \mathbf{B}(\ell^2)$ is endowed with a continuous trace-scaling action, in particular $\mathcal{F}(N) = \R_+$.
\end{itemize}
\end{corollaryE}

We recall in Section $\ref{preliminaries}$ a number of known results needed in
the proofs, for the reader's convenience. This includes a discussion of intertwining techniques for semifinite von Neumann algebras as well as several facts on the noncommutative flow of weights, Cartan subalgebras and the complete metric approximation property. Theorem A is proven in Section $\ref{theoA}$, Theorems B and D in Section $\ref{theoBD}$.

\section{Preliminaries}\label{preliminaries}

\subsection{Intertwining techniques}

We first recall some notation. Let $P \subset \mathcal{M}$ be an inclusion of von Neumann algebras. The {\it normalizer of} $P$ {\it inside} $\mathcal{M}$ is defined as
\begin{equation*}
\mathcal{N}_\mathcal{M}(P) := \left\{ u \in \mathcal{U}(\mathcal{M}) : \Ad(u) P = P \right\},
\end{equation*}
where $\Ad(u) = u \cdot u^*$. The inclusion $P \subset \mathcal{M}$ is said to be {\it regular} if $\mathcal{N}_\mathcal{M}(P)'' = \mathcal{M}$. The {\it groupoid normalizer of} $P$ {\it inside} $\mathcal{M}$ is defined as
\begin{equation*}
\mathcal{G}\mathcal{N}_\mathcal{M}(P) := \left\{ v \in \mathcal{M} \mbox{ partial isometry } : v P v^* \subset P, v^* P v \subset P \right\}.
\end{equation*}
The {\it quasi-normalizer of} $P$ {\it inside} $\mathcal{M}$ is defined as
\begin{equation*}
\QN_\mathcal{M}(P) := \left\{ a \in \mathcal{M} : \exists b_1, \dots, b_n \in \mathcal{M}, aP \subset \sum_i Pb_i, Pa \subset \sum_i b_iP \right\}.
\end{equation*}
The inclusion $P \subset \mathcal{M}$ is said to be {\it quasi-regular} if $\QN_\mathcal{M}(P)'' = \mathcal{M}$. Moreover,
\begin{equation*}
P' \cap \mathcal{M} \subset \mathcal{N}_\mathcal{M}(P)'' \subset \mathcal{G}\mathcal{N}_\mathcal{M}(P)'' \subset \QN_\mathcal{M}(P)''.
\end{equation*}

In \cite[Theorem 2.1]{popamal1}, \cite[Theorem A.1]{popa2001}, Popa introduced a powerful tool to prove the unitary conjugacy of two von Neumann subalgebras of a tracial von Neumann algebra $(M, \tau)$. We will make intensively use of this technique. If $A, B \subset (M, \tau)$ are (possibly non-unital) von Neumann subalgebras, denote by $1_A$ (resp. $1_B$) the unit of $A$ (resp. $B$).

\begin{theo}[Popa, \cite{{popamal1}, {popa2001}}]\label{intertwining1}
Let $(M, \tau)$ be a finite von Neumann algebra. Let $A, B \subset M$ be possibly nonunital von Neumann subalgebras. The following are equivalent:
\begin{enumerate}
\item There exist $n \geq 1$, a possibly nonunital $\ast$-homomorphism $\psi : A \to \mathbf{M}_n(\C) \otimes B$ and a nonzero partial isometry $v \in \mathbf{M}_{1, n}(\C) \otimes 1_AM1_B$ such that $x v = v \psi(x)$, for any $x \in A$.

\item There is no sequence of unitaries $(u_k)$ in $A$ such that 
\begin{equation*}
\lim_{k \to \infty} \|E_B(a^* u_k b)\|_2 = 0, \forall a, b \in 1_A M 1_B.
\end{equation*}
\end{enumerate}
\end{theo}
If one of the previous equivalent conditions is satisfied, we shall say that $A$ {\it embeds into} $B$ {\it inside} $M$ and denote $A \preceq_M B$. For simplicity, we shall write $M^n := \mathbf{M}_n(\C) \otimes M$.

In this paper, we will need to extend Popa's intertwining techniques to {\em semifinite} von Neumann algebras. Let $(M, \Tr)$ be a von Neumann algebra endowed with a semifinite faithful normal trace. We shall simply denote by $L^2(M)$ the $M, M$-bimodule $L^2(M, \Tr)$, and by $\|\cdot\|_{2, \Tr}$ the $L^2$-norm associated with $\Tr$. We will use the following well-known inequality ($\|\cdot\|_\infty$ is the operator norm):
\begin{equation*}
\|x \xi y\|_{2, \Tr} \leq \|\xi\|_{2, \Tr} \|x\|_\infty \|y\|_\infty, \forall \xi \in L^2(M), \forall x, y \in M. 
\end{equation*}
We shall say that a projection $p \in M$ is $\Tr$-{\it finite} if $\Tr(p) < \infty$. Then $p$ is necessarily finite. Moreover, $pMp$ is a finite von Neumann algebra and $\tau:= \Tr(p \cdot p)/\Tr(p)$ is a faithful normal tracial state on $pMp$. Recall that for any projections $p, q \in M$, we have $p \vee q - p \sim q - p \wedge q$. Then it follows that for any $\Tr$-finite projections $p, q \in M$, $p \vee q$ is still $\Tr$-finite and $\Tr(p \vee q) = \Tr(p) + \Tr(q) - \Tr(p \wedge q)$. 

Note that if a sequence $(x_k)$ in $M$ converges to $0$ $\ast$-strongly, then for any nonzero $\Tr$-finite projection $q \in M$, $\left\| x_k q \right\|_{2, \Tr} + \| q x_k \|_{2, \Tr} \to 0$. Indeed,
\begin{eqnarray*}
x_k \to 0 \ast-\mbox{strongly in } M & \Longleftrightarrow & x^*_k x_k + x_k x_k^* \to 0 \mbox{ weakly in } M \\
& \Longrightarrow & qx^*_kx_kq  + q x_k x_k^* q \to 0 \mbox{ weakly in } qMq \\
& \Longrightarrow & \Tr(q x^*_k x_k q) + \Tr(q x_k x_k^* q) \to 0 \\
& \Longleftrightarrow & \Tr((x_k q)^* (x_k q)) + \Tr((q x_k)^* q x_k) \to 0 \\
& \Longleftrightarrow & \| x_k q \|_{2, \Tr} + \|q x_k\|_{2, \Tr} \to 0.
\end{eqnarray*}
Moreover, there always exists an increasing sequence of $\Tr$-finite projections $(p_k)$ in $M$ such that $p_k \to 1$ strongly.

Intertwining techniques for semifinite von Neumann algebras were developed in \cite{houdayer4}. The following result due to S. Vaes is a slight improvement of \cite[Theorem 2.2]{houdayer4} that will be useful in the sequel. We thank S. Vaes for allowing us to present his proof.

\begin{lem}[Vaes, \cite{personal}]\label{intertwining}
Let $(M, \Tr)$ be a semifinite von Neumann algebra. Let $B \subset M$ be a von Neumann subalgebra such that $\Tr_{|B}$ is still semifinite. Let $p \in M$ be a nonzero projection such that $\Tr(p) < \infty$ and $A \subset pMp$ a von Neumann subalgebra. Then the following are equivalent:
\begin{enumerate}
\item For every nonzero projection $q \in B$ with $\Tr(q) < \infty$, we have
\begin{equation*}
A \npreceq_{eMe} qBq, \mbox{ where } e = p \vee q,
\end{equation*}
in the usual sense for finite von Neumann algebras.

\item There exists a sequence of unitaries $(u_n)$ in $A$ such that
\begin{equation*}
\lim_n \|E_B(x^* u_n y)\|_{2, \Tr} = 0, \forall x, y \in M.
\end{equation*}
\end{enumerate}
If these conditions hold, we write $A \npreceq_M B$ and otherwise we write $A \preceq_M B$.
\end{lem}

\begin{proof}
We prove both directions.

$(1) \Longleftarrow (2)$. Take a nonzero projection $q \in B$ such that $\Tr(q) < \infty$ and set $e = p \vee q$. Write $\lambda = \Tr(e)$. For all $x, y \in pMq$, using the $\| \cdot \|_2$-norm with respect to the normalized trace on $eMe$, we have
\begin{equation*}
\|E_{qBq}(x^* u_n y)\|_2 = \lambda^{-1/2} \|E_B(x^* u_n y)\|_{2, \Tr} \to 0.
\end{equation*}
This means exactly that $A \npreceq_{eMe} qBq$.

$(1) \Longrightarrow (2)$. Let $(q_n)$ be an increasing sequence of projections in $B$ such that $q_n \to 1$ strongly and $\Tr(q_n) < \infty$. Set $e_n = p \vee q_n$. Let $\{x_k : k \in \N\}$ be a $\ast$-strongly dense subset of $(M)_1$ (the unit ball of $M$). Since $A \npreceq_{e_n M e_n} q_n B q_n$, we can take a unitary $u_n \in \mathcal{U}(A)$ such that 
\begin{equation*}
\|E_B(q_n x_i u_n x_j q_n)\|_{2, \Tr} < \frac1n, \forall 1 \leq i, j \leq n.
\end{equation*}
Note that $u_n = p u_n p$.

Let $\varepsilon > 0$ and fix $x, y \in (M)_1$.  Since $q_m \to 1$ strongly and since $\Tr(p) < \infty$, take $m \in \N$ large enough such that 
\begin{equation*}
\| q_m xp - xp \|_{2, \Tr} + \| pyq_m - py \|_{2, \Tr} < \varepsilon.
\end{equation*}
Since $\Tr(q_m) < \infty$, next choose $i, j \in \N$ such that 
\begin{equation*}
\| q_m x p - q_m x_i \|_{2, \Tr} + \| pyq_m - x_j q_m \|_{2, \Tr} < \varepsilon.
\end{equation*}
Now, for every $n \in \N$, we have
\begin{eqnarray*}
\|E_B(x u_n y)\|_{2, \Tr} & = & \|E_B(x p u_n p y)\|_{2, \Tr} \\
& \leq & \|E_B(q_m x p u_n py q_m)\|_{2, \Tr} + \varepsilon \\
& \leq & \|E_B(q_m x_i u_n x_j q_m)\|_{2, \Tr} + 2 \varepsilon. 
\end{eqnarray*}
Therefore, if $n \geq \max\{m, i, j\}$, we get
\begin{equation*}
\|E_B(x u_n y)\|_{2, \Tr} \leq \frac1n + 2 \varepsilon.
\end{equation*}
\end{proof}

Write $\Tr_n$ for the non-normalized faithful trace on $\mathbf{M}_n(\C)$. The faithful normal semifinite trace $\Tr_n \otimes \Tr$ on $\mathbf{M}_n(\C) \otimes M$ will be simply denoted by $\Tr$. Observe that if $A \preceq_M B$ in the sense of Lemma $\ref{intertwining}$, then there exist $n \geq 1$, a nonzero projection $q \in B^n$ such that $\Tr(q) < \infty$, a nonzero partial isometry $v \in \mathbf{M}_{1, n}(\C) \otimes M$ and a unital $\ast$-homomorphism $\psi : A \to qB^nq$ such that $xv = v\psi(x)$, $\forall x\in A$. In the case when $A$ and $B$ are maximal abelian, one can get a more precise result. This is an analog of a result by Popa \cite[Theorem A.1]{popa2001} for semifinite von Neumann algebras.

\begin{prop}\label{intertwining-masa}
Let $(M, \Tr)$ be a semifinite von Neumann algebra. Let $B \subset M$ be a maximal abelian von Neumann subalgebra such that $\Tr_{|B}$ is still semifinite. Let $p \in M$ be a non-zero projection such that $\Tr(p) < \infty$ and $A \subset pMp$ a maximal abelian von Neumann subalgebra. The following are equivalent:
\begin{enumerate}
\item $A \preceq_M B$ in the sense of Lemma $\ref{intertwining}$.
\item There exists a nonzero partial isometry $v \in M$ such that $vv^* \in A$, $v^*v \in B$ and $v^* A v = B v^*v$.
\end{enumerate}
\end{prop}

\begin{proof}
We only need to prove $(1) \Longrightarrow (2)$. The proof is very similar to the one of \cite[Theorem A.1]{popa2001}. We will use exactly the same reasoning as in the proof of \cite[Theorem C.3]{vaesbern}.

Since $A \preceq_M B$ in the sense of Lemma $\ref{intertwining}$, we can find $n \geq 1$, a nonzero $\Tr$-finite projection $q \in \mathbf{M}_n(\C) \otimes B$, a nonzero partial isometry $w \in \mathbf{M}_{1, n}(\C) \otimes pM$ and a unital $\ast$-homomorphism $\psi : A \to q(\mathbf{M}_n(\C) \otimes B)q$ such that $xw = w \psi(x)$, $\forall x \in A$. Since we can replace $q$ by an equivalent projection in $\mathbf{M}_n(\C) \otimes B$, we may assume $q = \Diag_n(q_1, \dots, q_n)$ (see for instance second item in \cite[Lemma C.2]{vaesbern}). Observe now that $\Diag_n(q_1B, \dots, q_nB)$ is maximal abelian in $q(\mathbf{M}_n(\C) \otimes B)q$. Since $B$ is abelian, $q(\mathbf{M}_n(\C) \otimes B)q$ is of finite type ${\rm I}$. Since $A$ is abelian, up to unitary conjugacy by a unitary in $q(\mathbf{M}_n(\C) \otimes B)q$, we may assume that $\psi(A)Ê\subset \Diag_n(q_1B, \dots, q_nB)$ (see \cite[Lemma C.2]{vaesbern}). We can now cut down $\psi$ and $w$ by one of projections $(0, \dots, q_i, \dots, 0)$ and assume $n = 1$ from the beginning.

Write $e = ww^* \in A$ (since $A' \cap pMp = A$) and $f = w^*w \in \psi(A)' \cap qMq$. By spatiality, we have
$$f(\psi(A)' \cap qMq)f = (\psi(A)f)' \cap fMf = (w^*Aw)' \cap fMf = w^* A w,$$
which is abelian. Let $Q := \psi(A)' \cap qMq$, which is a finite von Neumann algebra. Since $Bq \subset Q$ is maximal abelian and $f \in Q$ is an abelian projection, \cite[Lemma C.2]{vaesbern} yields a partial isometry $u \in Q$ such that $uu^* = f$ and $u^*Qu \subset Bq$. Define now $v = wu$. We get 
$$v^* A v = u^* w^* A w u = u^* f(\psi(A)' \cap qMq)f u \subset Bq.$$
Moreover $vv^* = wuu^*w^* = wfw^* = e \in A$. Since $v^*A v$ and $Bv^*v$ are both maximal abelian, we get $v^* A v = B v^*v$.
\end{proof}

\subsection{The noncommutative flow of weights}

Let $\mathcal{M}$ be a von Neumann algebra. Let $\varphi$ be a faithful normal state on $\mathcal{M}$. Denote by $\mathcal{M}^\varphi$ the centralizer and by $M = \mathcal{M} \rtimes_{\sigma^\varphi} \R$ the {\it core} of $\mathcal{M}$, where $\sigma^\varphi$ is the modular group associated with the state $\varphi$. Denote by $\pi_{\sigma^\varphi} : \mathcal{M} \to M$ the representation of $\mathcal{M}$ in its core $M$, i.e.\ $\pi_{\sigma^\varphi}(x) = (\sigma_{-t}^\varphi(x))_{t \in \R}$ for every $x \in \mathcal{M}$, and denote by $\lambda^\varphi(s)$ the unitaries in $L(\R)$ implementing the action $\sigma^\varphi$. Consider the {\it dual weight} $\widehat{\varphi}$ on $M$ (see \cite{takesaki73}) which satisfies the following:
\begin{eqnarray*}
\sigma_t^{\widehat{\varphi}}(\pi_{\sigma^\varphi}(x)) & = & \pi_{\sigma^\varphi}(\sigma_t^\varphi(x)), \forall x \in \mathcal{M} \\
\sigma_t^{\widehat{\varphi}}(\lambda^\varphi(s)) & = & \lambda^\varphi(s), \forall s \in \R.
\end{eqnarray*} 
Note that $\widehat{\varphi}$ is a semifinite faithful normal weight on $M$. Write $\theta^\varphi$ for the dual action of $\sigma^\varphi$ on $M$, where we identify $\R$ with its Pontryagin dual. Take now $h_\varphi$ a nonsingular positive self-adjoint operator affiliated with $L(\R)$ such that $h_\varphi^{is} = \lambda^\varphi(s)$, for any $s \in \R$. Define $\Tr_\varphi := \widehat{\varphi}(h_\varphi^{-1} \cdot)$. We get that $\Tr_\varphi$ is a semifinite faithful normal trace on $M$ and the dual action $\theta^\varphi$ {\it scales} the trace $\Tr_\varphi$:
\begin{equation*}
\Tr_\varphi \circ \theta^\varphi_s(x) = e^{-s} \Tr_\varphi(x), \forall x \in M_+, \forall s \in \R.
\end{equation*}
Moreover, the canonical faithful normal conditional expectation $E_{L(\R)} : M \to L(\R)$ defined by $E_{L(\R)}(x \lambda^\varphi(s)) = \varphi(x) \lambda^\varphi(s)$ preserves the trace $\Tr_\varphi$, i.e.
\begin{equation*}
\Tr_\varphi \circ E_{L(\R)} (x) = \Tr_\varphi(x), \forall x \in M_+.
\end{equation*}

There is also a functorial construction of the core of the von Neumann algebra $\mathcal{M}$ which does not rely on the choice of a particular state $\varphi$ on $\mathcal{M}$ (see \cite{{connes73}, {connestak}, {falcone}}). This is called the {\it noncommutative flow of weights}. We will simply denote it by $(\mathcal{M} \subset M, \theta, \Tr)$, where $M$ is the {\em core} of $\mathcal{M}$, $\theta$ is the dual action of $\R$ on the core $M$ and $\Tr$ is the semifinite faithful normal trace on $M$ such that $\Tr \circ \theta_s = e^{-s}\Tr$, for any $s \in \R$. Let $\varphi$ be a faithful normal state on $\mathcal{M}$. It follows from \cite[Theorem 3.5]{falcone} and \cite[Theorem ${\rm XII.6.10}$]{takesakiII} that there exists a natural $\ast$-isomorphism
\begin{equation*}
\Pi_{\varphi} : \mathcal{M} \rtimes_{\sigma^\varphi} \R \to M
\end{equation*}
such that 
\begin{eqnarray*}
\Pi_{\varphi} \circ \theta^\varphi & = & \theta \circ \Pi_{\varphi} \\
\Tr_\varphi & = & \Tr \circ \Pi_{\varphi} \\
\Pi_{\varphi}(\pi_{\sigma^\varphi}(\mathcal{M})) & = & \mathcal{M}.
\end{eqnarray*}
Let now $\varphi, \psi$ be two faithful normal states on $\mathcal{M}$. Through the $\ast$-isomorphism $\Pi_{\varphi, \psi} := \Pi_\psi^{-1} \circ \Pi_\varphi : \mathcal{M} \rtimes_{\sigma^\varphi} \R \to \mathcal{M} \rtimes_{\sigma^\psi} \R$, we will identify 
$$(\pi_{\sigma^\varphi}(\mathcal{M}) \subset \mathcal{M} \rtimes_{\sigma^\varphi} \R, \theta^\varphi, \Tr_\varphi) \mbox{ with } (\pi_{\sigma^\psi}(\mathcal{M}) \subset \mathcal{M} \rtimes_{\sigma^\psi} \R, \theta^\psi, \Tr_\psi).$$ 
In the sequel, we will refer to the triple $(\mathcal{M} \subset M, \theta, \Tr)$ as the noncommutative flow of weights. By Takesaki's Duality Theorem \cite{takesaki73}, we have
$$(\mathcal{M} \rtimes_\sigma \R) \rtimes_{(\theta_s)} \R \cong \mathcal{M} \overline{\otimes} \mathbf{B}(L^2(\R)).$$
In particular, $\mathcal{M}$ is amenable if and only if $M = \mathcal{M} \rtimes_\sigma \R$ is amenable. The following well-known proposition will be useful for us.

\begin{prop}
Let $\varphi$ be a faithful normal state on $\mathcal{M}$. Let $M =  \mathcal{M} \rtimes_{\sigma^{\varphi}} \R$ be as above. Then $L(\R)' \cap M = \mathcal{M}^\varphi \overline{\otimes} L(\R)$. In particular, if $\mathcal{M}^\varphi = \C$ then $L(\R)$ is maximal abelian in $M$.
\end{prop}
\begin{proof}
We regard $M = \mathcal{M} \rtimes_{\sigma^\varphi} \R$ generated by $\pi(x) = (\sigma_{-t}^\varphi(x))_{t \in \R}$, for $x \in \mathcal{M}$, and $1 \otimes \lambda^\varphi(t)$, for $t \in \R$. Therefore $M \subset \mathcal{M} \overline{\otimes} \mathbf{B}(L^2(\R))$. Since $L(\R) \subset \mathbf{B}(L^2(\R))$ is maximal abelian, we get $L(\R)' \cap M \subset \mathcal{M} \overline{\otimes} L(\R)$.

Denote by $\widehat{\varphi}$ the dual weight of $\varphi$ on $M$ (see e.g.\ \cite{takesaki73}). The following relations are true: for every $s, t \in \R$, for every $x \in \mathcal{M}$,
\begin{eqnarray*}
\sigma_t^{\widehat{\varphi}}(\pi(x)) & = & \pi(\sigma^\varphi_t(x)) \\
\sigma_t^{\widehat{\varphi}}(1 \otimes \lambda^\varphi(s)) & = & 1 \otimes \lambda^\varphi(s) \\
\Delta_{\widehat{\varphi}}^{it} & = & \Delta_\varphi^{it} \otimes 1.
\end{eqnarray*}
Since $(1 \otimes \lambda^\varphi(s))_{s \in \R}$ is a $1$-cocycle for $(\sigma_t^{\widehat{\varphi}})$, \cite[Th\'eor\`eme 1.2.4]{connes73} implies that the faithful normal semifinite weight $\Tr$ given by $\sigma_t^{\Tr} = (1 \otimes \lambda^\varphi(t))^* \sigma_t^{\widehat{\varphi}} (1 \otimes \lambda^\varphi(t))$ is a trace on $M$. This implies that $L(\R)' \cap M$ is exactly the centralizer of the weight $\widehat{\varphi}$. Since $\Delta_{\widehat{\varphi}}^{it}  =  \Delta_\varphi^{it} \otimes 1$, for every $t \in \R$, we get $L(\R)' \cap M \subset \mathcal{M}^\varphi \overline{\otimes} \mathbf{B}(L^2(\R))$. Thus $L(\R)' \cap M = \mathcal{M}^\varphi \overline{\otimes} L(\R)$.
\end{proof}

\subsection{Basic facts on Cartan subalgebras}

\begin{df}
Let $\mathcal{M}$ be any von Neumann algebra. A von Neumann subalgebra $A \subset \mathcal{M}$ is said to be a {\em Cartan subalgebra} if the following conditions hold:
\begin{enumerate}
\item $A$ is maximal abelian, i.e. $A = A' \cap \mathcal{M}$.
\item There exists a faithful normal conditional expectation $E : \mathcal{M} \to A$.
\item The normalizer $\mathcal{N}_{\mathcal{M}}(A) = \{u \in \mathcal{U}(\mathcal{M}) : u A u^* = A\}$ generates $\mathcal{M}$.
\end{enumerate}
\end{df}

Let $A \subset \mathcal{M}$ be a Cartan subalgebra. Let $\tau$ be a faithful normal tracial state on $A$. Then $\varphi = \tau \circ E$ is a faithful normal state on $\mathcal{M}$. Moreover $A \subset \mathcal{M}^\varphi$, where $\mathcal{M}^\varphi$ denotes the centralizer of $\varphi$. Write $(\sigma_t^\varphi)$ for the modular automorphism group. Denote by $M = \mathcal{M} \rtimes_{\sigma^\varphi} \R$ the continuous core and write $\lambda^\varphi(t)$ for the unitaries in $M$ which implement the modular action. The following  proposition is well-known and will be a crucial tool in order to prove Theorem B. We include a proof for the reader's convenience.

\begin{prop}
The von Neumann subalgebra $A \overline{\otimes} L(\R) \subset \mathcal{M} \rtimes_{\sigma^\varphi} \R$ is a Cartan subalgebra.
\end{prop}

\begin{proof}
Since $A \subset \mathcal{M}$ and $L(\R) \subset \mathbf{B}(L^2(\R))$ are both maximal abelian, it follows that $A \overline{\otimes} L(\R)$ is maximal abelian in $\mathcal{M} \overline{\otimes} \mathbf{B}(L^2(\R))$. Therefore $A \overline{\otimes} L(\R)$ is maximal abelian in $\mathcal{M} \rtimes_{\sigma^\varphi} \R$. 

The faithful normal conditional expectation $F : \mathcal{M} \rtimes_{\sigma^\varphi} \R \to A \overline{\otimes} L(\R)$ is given by: $F(x \lambda^\varphi(t)) = E(x) \lambda^\varphi(t)$, $\forall x \in \mathcal{M}, \forall t \in \R$. Observe that $F$ preserves the canonical trace $\Tr_\varphi$.

It remains to show that $A \overline{\otimes} L(\R)$ is regular in $\mathcal{M} \rtimes_{\sigma^\varphi} \R$. Recall that $A \subset \mathcal{M}^\varphi$, so that $a \lambda^\varphi(t) = \lambda^\varphi(t) a$, for every $t \in \R$ and every $a \in A$. For every $t \in \R$, every $u \in \mathcal{N}_{\mathcal{M}}(A)$ and every $a \in A$, we have
\begin{eqnarray*}
\sigma_t^\varphi(u) u^* a & = & \sigma_t^\varphi(u) (u^* a u) u^* \\
& = & \sigma_t^\varphi(u u^* a u) u^* \\
& = & a \sigma_t^\varphi(u) u^*,
\end{eqnarray*}
so that $\sigma_t^\varphi(u)u^* \in A' \cap \mathcal{M} = A$. We moreover have
\begin{equation*}
u(a \lambda^\varphi(t))u^* = (uau^*) u\lambda^\varphi(t) = (ua u^*) (u \sigma^\varphi_t(u^*)) \lambda^\varphi(t),
\end{equation*}
so that $u (A \overline{\otimes} L(\R)) u^* = A \overline{\otimes} L(\R)$. Consequently, $A \overline{\otimes} L(\R) \subset \mathcal{M} \rtimes_{\sigma^\varphi} \R$ is regular.
\end{proof}

Assume that $\mathcal{M}$ is a type ${\rm II}$ von Neumann algebra. Then $M = \mathcal M \rtimes_\sigma \R$ is still of type ${\rm II}$. Assume now that $\mathcal{M}$ is a type ${\rm III}$ von Neumann algebra. Then $M$ is of type ${\rm II_\infty}$. Let $p \in A \overline{\otimes} L(\R)$ be a nonzero projection such that $\Tr(p) < \infty$, so that $pMp$ is of type ${\rm II_1}$. The next proposition shows that $(A \overline{\otimes} L(\R))p \subset pMp$ is a Cartan subalgebra.

\begin{prop}\label{masa}
Let $N$ be a type ${\rm II_\infty}$ von Neumann algebra with a faithful normal semifinite trace $\Tr$. Let $B \subset N$ be a maximal abelian $\ast$-subalgebra for which $\Tr_{|B}$ is still semifinite. Let $p \in B$ be a nonzero projection such that $\Tr(p) < \infty$. Then $\mathcal{N}_{pMp}(Bp)'' = p\mathcal{N}_M(B)''p$.
\end{prop}

\begin{proof}
The equality $(pBp)' \cap pMp = p(B' \cap M)p$ is well-known (see for instance \cite[Lemma 2.1]{masapopa}). Thus, $Bp$ is maximal abelian in $pMp$. Let $u \in \mathcal{N}_{M}(B)$. We have
\begin{equation*}
pup (Bp) = puBp = p Bu p = (Bp) pup.
\end{equation*}
It follows that $p\mathcal{N}_M(B)''p \subset \mathcal{QN}_{pMp}(Bp)''$. The normalizer and the quasi-normalizer of a maximal abelian subalgebra generate the same von Neumann algebra (see \cite[Theorem 2.7]{popa-shlyakhtenko}). Thus $\mathcal{QN}_{pMp}(Bp)'' = \mathcal{N}_{pMp}(Bp)''$ and $p\mathcal{N}_M(B)''p \subset \mathcal{N}_{pMp}(Bp)''$. Let now $v \in \mathcal{N}_{pMp}(Bp)$. Define $u = v + (1 - p) \in \mathcal{U}(M)$. It is clear that $u \in \mathcal{N}_M(B)$ and $pup = v$. Therefore $\mathcal{N}_{pMp}(Bp)'' \subset p\mathcal{N}_M(B)''p$, which finishes the proof.
\end{proof}

\subsection{Complete metric approximation property}

\begin{df}[Haagerup, \cite{haa}]
A von Neumann algebra $\mathcal{N}$ is said to have the ({\em weak}$^*$) {\it complete bounded approximation property} if there exist a constant $C \geq 1$ and a net of normal finite rank completely bounded maps $\Phi_n : \mathcal{N} \to \mathcal{N}$ such that
\begin{itemize}
\item $\Phi_n(x) \to x$ $\ast$-strongly, for every $x \in \mathcal{N}$;
\item $\limsup _n\|\Phi_n\|_{\cb} \leq C$.
\end{itemize}
The Cowling-Haagerup constant $\Lambda_{\cb}(\mathcal{N})$ is defined as the infimum of the constants $C$ for which a net $(\Phi_n)$ as above exists. Also we say that $\mathcal{N}$ has the ({\em weak}$^*$) {\it complete metric approximation property} (c.m.a.p.) if $\Lambda_{\cb}(\mathcal{N}) = 1$.
\end{df}

\begin{theo}\label{properties}
The following are true.
\begin{enumerate}
\item $\Lambda_{\cb}(p \mathcal{M} p) \leq \Lambda_{\cb}(\mathcal{M})$, for every projection $p \in \mathcal{M}$.
\item If $\mathcal{N} \subset \mathcal{M}$ such that there exists a conditional expectation $E :  \mathcal{M} \to \mathcal{N}$, then $\Lambda_{\cb}(\mathcal{N}) \leq \Lambda_{\cb}(\mathcal{M})$.
\item If $\mathcal{M}$ is amenable then $\Lambda_{\cb}(\mathcal{M}) = 1$.
\item Denote by $\sigma$ the modular automorphism group on $\mathcal{M}$. Then $
\Lambda_{\cb}(\mathcal{M}) = \Lambda_{\cb}(\mathcal{M} \rtimes_\sigma \R)$.
\item If $\mathcal{M}_i$ is amenable for every $i \in I$, then $\Lambda_{\cb}( \ast_{i \in I}\mathcal{M}_i) = 1$.
\end{enumerate}
\end{theo}

\begin{proof}
$(1), (2), (4)$ follow from \cite{anan95}. The equivalence between semidiscreteness and amenability \cite{connes76} gives $(3)$. Finally $(5)$ is due to \cite{RicardXu}.
\end{proof}

\subsection{Free Araki-Woods factors}\label{freearak}

Recall now the construction of the free Araki-Woods factors due to Shlyakhtenko \cite{shlya97}. Let $H_{\R}$ be a real separable Hilbert space and let $(U_t)$ be an orthogonal representation of $\R$ on $H_{\R}$. Let $H = H_{\R} \otimes_{\R} \C$ be the complexified Hilbert space. Let $J$ be the canonical anti-unitary involution on $H$ defined by:
\begin{equation*}
J(\xi + i \eta) = \xi - i \eta, \forall \xi, \eta \in H_\R.
\end{equation*}
 If $A$ is the infinitesimal generator of $(U_t)$ on $H$, we recall that $j : H_{\R} \to H$ defined by $j(\zeta) = (\frac{2}{A^{-1} + 1})^{1/2}\zeta$ is an isometric embedding of $H_{\R}$ into $H$. Moreover, we have $JAJ = A^{-1}$. Let $K_{\R} = j(H_{\R})$. It is easy to see that $K_\R \cap i K_\R = \{0\}$ and $K_\R + i K_\R$ is dense in $H$. Write $I = J A^{-1/2}$. Then $I$ is a conjugate-linear closed invertible operator on $H$ satisfying $I = I^{-1}$ and $I^*I = A^{-1}$. Such an operator is called an {\it involution} on $H$. Moreover, $K_\R = \{ \xi \in \dom(I) : I \xi = \xi \}$.

We introduce the \emph{full Fock space} of $H$:
\begin{equation*}
\mathcal{F}(H) =\C\Omega \oplus \bigoplus_{n = 1}^{\infty} H^{\otimes n}.
\end{equation*}
The unit vector $\Omega$ is called the \emph{vacuum vector}. For any $\xi \in H$, define the {\it left creation} operator $\ell(\xi) : \mathcal{F}(H) \to \mathcal{F}(H)$
\begin{equation*}
\left\{ 
{\begin{array}{l} \ell(\xi)\Omega = \xi, \\ 
\ell(\xi)(\xi_1 \otimes \cdots \otimes \xi_n) = \xi \otimes \xi_1 \otimes \cdots \otimes \xi_n.
\end{array}} \right.
\end{equation*}
We have $\|\ell(\xi)\|_\infty = \|\xi\|$ and $\ell(\xi)$ is an isometry if $\|\xi\| = 1$. For any $\xi \in H$, we denote by $s(\xi)$ the real part of $\ell(\xi)$ given by
\begin{equation*}
s(\xi) = \frac{\ell(\xi) + \ell(\xi)^*}{2}.
\end{equation*}
The crucial result of Voiculescu \cite{voiculescu92} is that the distribution of the operator $s(\xi)$ with respect to the vacuum vector state $\chi(x) = \langle x\Omega, \Omega\rangle$ is the semicircular law of Wigner supported on the interval $[-\|\xi\|, \|\xi\|]$. 

\begin{df}[Shlyakhtenko, \cite{shlya97}]
Let $(U_t)$ be an orthogonal representation of $\R$ on the real Hilbert space $H_{\R}$. The \emph{free Araki-Woods} von Neumann algebra associated with $(H_\R, U_t)$, denoted by $\Gamma(H_{\R}, U_t)''$, is defined by
\begin{equation*}
\Gamma(H_{\R}, U_t)'' := \{s(\xi) : \xi \in K_{\R}\}''.
\end{equation*}
We will denote by $\Gamma(H_{\R}, U_t)$ the $C^*$-algebra generated by the $s(\xi)$'s for all $\xi \in K_\R$.
\end{df}

The vector state $\chi(x) = \langle x\Omega, \Omega\rangle$ is called the {\it free quasi-free state} and is faithful on $\Gamma(H_\R, U_t)''$. Let $\xi, \eta \in K_\R$ and write $\zeta = \xi + i \eta$. We have
\begin{equation*}
2 s(\xi) + 2 i s(\eta) = \ell(\zeta) + \ell(I \zeta)^*.
\end{equation*}
Thus, $\Gamma(H_\R, U_t)''$ is generated as a von Neumann algebra by the operators of the form $\ell(\zeta) + \ell(I \zeta)^*$ where $\zeta \in \dom(I)$. Note that the modular group $(\sigma_t^\chi)$ of the free quasi-free state $\chi$ is given by $\sigma^{\chi}_{- t} = \Ad(\mathcal{F}(U_t))$, where $\mathcal{F}(U_t) = 1 \oplus \bigoplus_{n \geq 1} U_t^{\otimes n}$. In particular, it satisfies
\begin{equation*}
\sigma_{-t}^{\chi}\left( \ell(\zeta) + \ell(I \zeta)^* \right)  =  \ell(U_t \zeta) + \ell(I U_t \zeta)^*, \forall \zeta \in \dom(I), \forall t \in \R. 
\end{equation*}

The free Araki-Woods factors provided many new examples of full factors of type {\rm III} \cite{{barnett95}, {connes73}, {shlya2004}}. We can summarize the general properties of the free Araki-Woods factors in the following theorem (see also \cite{vaes2004}):

\begin{theo}[Shlyakhtenko, \cite{{shlya2004}, {shlya99}, {shlya98}, {shlya97}}]
Let $(U_t)$ be an orthogonal representation of $\R$ on the real Hilbert space $H_{\R}$ with $\dim H_{\R} \geq 2$. Denote by $\mathcal{M} := \Gamma(H_{\R}, U_t)''$.
\begin{enumerate}
\item $\mathcal{M}$ is a full factor and Connes' invariant $\tau(\mathcal{M})$ is the weakest topology on $\R$ that makes the map $t \mapsto U_t$ $\ast$-strongly continuous.
\item $\mathcal{M}$ is of type ${\rm II_1}$ if and only if $U_t = 1$, for every $t \in \R$.
\item $\mathcal{M}$ is of type ${\rm III_{\lambda}}$ $(0 < \lambda < 1)$ if and only if $(U_t)$ is periodic of period $\frac{2\pi}{|\log \lambda|}$.
\item $\mathcal{M}$ is of type ${\rm III_1}$ in the other cases.
\item The factor $\mathcal{M}$ has almost periodic states if and only if $(U_t)$ is almost periodic.
\end{enumerate}
\end{theo}
Shlyakhtenko moreover showed \cite{shlya2004} that every free Araki-Woods factor $\mathcal{M} = \Gamma(H_\R, U_t)''$ is {\em generalized solid} in the sense of \cite{{ozawa2003}, {VV}}: for every diffuse subalgebra $A \subset \mathcal{M}$ for which there exists a faithful normal conditional expectation $E : \mathcal{M} \to A$, the relative commutant $A' \cap \mathcal{M}$ is amenable. The first-named author showed \cite{houdayer5} that every type ${\rm III_1}$ free Araki-Woods factor has trivial bicentralizer \cite{haagerup84}.

\section{Approximation properties: proof of Theorem A}\label{theoA}

There are not so many ways to produce concrete examples of completely
bounded maps on free Araki-Woods von Neumann algebras. When $(U_t)$ is
trivial, one recovers the free group algebras, and harmonic analysis
joins the game with Fourier multipliers. On $\mathbf F_\infty$,
multipliers that only depend on the length are said to be {\em radial}.  Haagerup and
Szwarc obtained a very nice characterization of them.  Their
approach was based on a one-to-one correspondence between Fourier
multipliers on a group $G$ and Schur multipliers on
$\mathbf{B}(\ell^2(G))$ established by Gilbert. Their idea was to look
for a description of Schur multipliers obtained this way and they
managed to do so for more general multipliers related to homogeneous
trees. The key point is to find a shift algebra that is preserved by
those Schur multipliers. This technique or some variations have
operated with success on other groups \cite{HSS, Wys}.

The free semicircular random variables and the canonical generators of $\mathbf
F_\infty$ have different shape but there is a natural {\em length} for both of them 
which is related to freeness. This notion still makes sense after the
quasi-free deformation and one can hope to have nice multipliers. We follow the scheme of Haagerup and Szwarc, but Gilbert's
theorem is missing here (there is no easy way to extend multipliers).
Nevertheless, we obtain exactly the same characterization and the
parallel with Schur multipliers is very striking. This is the first
step towards the approximation property that originates from the paper
\cite{haa}, where it was shown that the projection onto tensors of a
fixed given length is bounded. Haagerup's ideas turned out to be 
efficient to prove various approximation properties in relationship
with Khintchine type inequalities (see \cite{Buc, Nou}). The second
step consists in using functorial completely positive maps called second 
quantizations (see \cite{shlya97, BSK}). The new point is that we show 
that the second quantization is valid under a milder assumption than the one
in \cite{shlya97}.

\subsection{Preliminaries}
The $C^*$-algebra $\Gamma(H_{\R}, U_t)$ is generated by real parts of some
left creation operators. Since we look for completely bounded maps on free
Araki-Woods algebras, it seems natural to try to find them as
restrictions on some larger algebra.  This is why we are interested in
basic properties of the algebra generated by creation operators.

 To fix notation, let $H$ be a complex Hilbert space and $\mathcal{F}(H)$
the corresponding full Fock space. We write $\Ch$ for the $C^*$-algebra
generated by all the left creation operators 
$\Ch = \langle \ell(e) : e \in H \rangle$. It is easy to verify that for any $e,f \in H$:
$$\ell(f)^*\ell(e)= \langle f,e \rangle.$$ 
In fact, this property
completely characterizes the algebra $\Ch$. Indeed, in the sense of
\cite{Pim}, $\Ch$ is a Toeplitz algebra and satisfies the following universal
property (see \cite[Theorem 3.4]{Pim}): if $u: H \to \mathbf{B}(K)$ is a linear map (for some Hilbert space $K$) so that $u^*(f)u(e)= \langle
f,e \rangle$, then there is a unique $*$-homomorphism $\pi : \Ch \to \mathbf{B}(K)$ so that $\pi(\ell(e))=u(e)$. 

When $H = \C$, we will
simply denote $\mathcal{T}(\C)$ by $\mathcal{T}$: this is the
universal $C^*$-algebra generated by a shift operator $S$ (a
nonunitary isometry).  We will need the following very elementary
estimates about creation operators:
\begin{lem}\label{majf}
For orthonormal families $(e_i), (f_i)$ in $H$ and $\alpha_i\in \C$ with $|\alpha_i| \leq 1$, we have
$$\left\|\frac 1 n\sum_{i=1}^n \alpha_i \ell(e_i)\ell(f_i)^{*}\right\|_\infty \leq \frac 1 n \qquad \textrm{  and }\qquad  \left\|\frac 1 n\sum_{i=1}^n \alpha_i \ell(e_i) \ell(f_i)\right\|_\infty \leq\frac 1 {\sqrt{n}}.$$
\end{lem}

\begin{proof}
Let $(e_i), (f_i)$ be orthonormal families in $H$ and $\alpha_i\in \C$ with $|\alpha_i| \leq 1$. The first inequality follows from
\begin{eqnarray*}
\left( \frac 1 n \sum_{i=1}^n \alpha_i \ell(e_i)\ell(f_i)^{*} \right)\left( \frac 1 n \sum_{i=1}^n \alpha_i \ell(e_i)\ell(f_i)^{*} \right)^* & = & \frac{1}{n^2} \sum_{i=1}^n |\alpha_i|^2 \ell(e_i)\ell(e_i)^{*} \\
& \leq & \frac{1}{n^2} \sum_{i = 1}^n  \ell(e_i)\ell(e_i)^{*} \\
& \leq & \frac{1}{n^2}.
\end{eqnarray*}
The second inequality follows from
\begin{eqnarray*}
\left( \frac 1 n \sum_{i=1}^n \alpha_i \ell(e_i)\ell(f_i) \right)^* \left( \frac 1 n \sum_{i=1}^n \alpha_i \ell(e_i)\ell(f_i) \right) & = & \frac{1}{n^2} \sum_{i=1}^n |\alpha_i|^2 \ell(f_i)^*\ell(f_i) \\
& \leq & \frac{1}{n^2} \sum_{i = 1}^n  \ell(f_i)^*\ell(f_i) \\
& = & \frac{1}{n}.
\end{eqnarray*}
\end{proof}

We come back to free Araki-Woods algebras as in \ref{freearak} with
the same notation: $\Gamma(H_{\R}, U_t)=\langle s(\xi) : \xi \in K_{
\R}\rangle$ is the $C^*$-algebra generated by the $s(\xi)$'s for all $\xi \in
K_\R$, and $\Gamma(H_{\R}, U_t)''$ is the corresponding von Neumann algebra.
Given any vector $e$ in $K_\R + i K_\R$, we will simply write $\overline e$ for $I(e)$ as $I(h+ik)=h-ik$, for $h,k \in K_\R$.

 The vacuum vector $\Omega$ is separating and cyclic for $\G$. Consequently any
$x \in \G$ is uniquely determined by $\xi= x\Omega\in \mathcal{F}(H)$, so we will write $x = W(\xi)$. Note that for $\xi \in K_\R$, we recover the semicircular random variables $W(\xi) = 2s(\xi)$ generating $\G$. It readily yields $W(e) = \ell(e) + \ell(\overline e)^*$, for every $e \in K_\R + i K_\R$.

 Given any vectors $e_{k}$ belonging to $K_\R + i
 K_\R$, it is easy to check that $e_1\otimes \cdots \otimes e_n$ lies in $\GG
 \Omega$. Moreover we have a nice description of $W(e_1\otimes \cdots \otimes e_n)$
 in terms of the $\ell(e_{k})$'s called the {\em Wick
 formula}. Since it plays a crucial role in our arguments, we state it as a lemma.
 
\begin{lem}[Wick formula]\label{Wick} For any $(e_{i})_{i \in \N}$ in $K_\R + i
 K_\R$ and any $n\geq 0$:
$$W(e_1\otimes \cdots\otimes e_n)= \sum_{k = 0}^n
\ell(e_{1}) \cdots \ell(e_{k})\ell(\overline e_{{k+1}})^* \cdots \ell(\overline
e_{n})^*.$$ 
\end{lem}
\begin{proof}
We prove it by induction on $n$. For $n = 0, 1$, we have $W(\Omega)=1$ 
and we observed that $W(e_i)=\ell(e_i) + \ell(\overline e_i)^*$.

Next, for $e_{0}\in K_\R+iK_\R$, we have 
\begin{eqnarray*}
W(e_{0})W(e_1\otimes \cdots \otimes e_n)\Omega&=&W(e_{0})(e_1\otimes \cdots \otimes e_n)
\\ &=& (\ell(e_{0})+\ell(\overline e_{0})^*)e_1\otimes \cdots \otimes e_n\\ &=& e_{0}\otimes e_1\otimes \cdots \otimes e_n + 
\langle \overline e_{0},e_{1}\rangle e_2\otimes \cdots \otimes e_n.
\end{eqnarray*}
Hence $$W(e_0\otimes \cdots \otimes e_n)=W(e_{0})W(e_1\otimes \cdots \otimes e_n) - \langle \overline e_{0},e_{1}\rangle W(e_2\otimes \cdots \otimes e_n),$$
but using the assumption for $n$ and $n-1$  and the commutation relations
$$\ell(\overline e_{0})^*W(e_1\otimes \cdots \otimes e_n)=\langle \overline
e_{0},e_{1}\rangle W(e_2\otimes \cdots \otimes e_n)+\ell(\overline
e_{0})^*\ell(\overline e_{1})^* \cdots \ell(\overline e_{n})^*.$$ Finally
$\ell(e_{0})W(e_1\otimes \cdots \otimes e_n)$ gives the first $n$ terms in
the Wick formula for order $n+1$.
\end{proof}

This formula expresses $W(e_1\otimes \cdots \otimes e_n)$ as an element of $\Ch$
and has many consequences such as Khintchine type inequalities in
\cite{Buc, Nou, BSK} for instance.  We let $$\mathcal{W}=\vect
\left\{W(e_1\otimes \cdots\otimes e_n) : n\geq 0,
e_{k} \in K_\R + i K_\R \right\}.$$ 
It is a dense $\ast$-subalgebra of $\GG$.

We will use the notion of completely bounded maps (see \cite{Pis}). We will not need 
very much beyond definitions and the fact that bounded 
functionals are automatically completely bounded (with the same norm).

\subsection{Radial multipliers}
The construction of our {\it radial multipliers} relies on some
functionals on $\mathcal{T}$. Let $\varphi: \N \to \C$ be a function. 
The {\em radial functional} $\gamma$ associated to $\varphi$ 
is defined on  $\vect \{S^iS ^{*j}\}\subset
\mathcal{T}$ by $\gamma (S^iS^{*j})=\varphi(i+j)$.

 The $C^*$-algebra $\mathcal T$ admits very few irreducible representations (the identity and its characters). It is thus possible to 
compute exactly the norm of such radial linear forms,
see \cite[Proposition 1.8 and Theorem
1.3]{HSS} and \cite{Wys}:
\begin{prop}\label{calcul}
The functional $\gamma$ extends to a bounded map on $\mathcal{T}$ if and only if 
$B = [\varphi(i+j) - \varphi(i+j+2)]_{i,j \geq 0}$
is a trace-class operator. If this is the case, then there are constants
$c_1,c_2 \in \C$ and a unique $\psi: \N \longrightarrow\C$ such that
$$\forall n \in \N, \varphi(n) = c_1 + c_2(-1)^n + \psi(n),
 \mbox{ and } \lim_n \psi(n)=0.$$ 
Moreover
$$\|\gamma\|_{\mathcal{T}^*}= |c_1| + |c_2| + \|B\|_1,$$
where $\|B\|_1$ is the trace norm of $B$.
\end{prop}
We say that $\gamma$ is the {\em radial} functional associated to $\varphi$. 
The definition of  multipliers on $\G$ follows the same scheme. 
Define $\m_\varphi$ on $\mathcal{W}$ by
$$\m_\varphi(W(e_1\otimes \cdots \otimes e_n))=\varphi(n) W(e_1\otimes \cdots \otimes e_n).$$

\begin{lem}\label{norext}
Let $\varphi: \N \to \C$ be any function. If $\m_\varphi$ can
be extended to a completely contractive map on $\GG$, then there is a
unique normal completely contractive extension of $\m_\varphi$ from $\G$ to $\G$.
\end{lem}

\begin{proof}
This is a standard fact. The space $\mathcal{W}$ is norm dense in $\GG$  which is 
weak-$*$ dense in $\G$. So $\mathcal{W}$ is also norm dense in $\G_*$ using the
basic embedding $\G \to \G_*$ given by $j(x)(y)=\chi(xy)$ (where $\chi$ denotes the free quasi-free state). By a duality argument, 
$\m_{\varphi} : \mathcal{W} \to \mathcal{W}$ extends uniquely to a completely 
contractive map on $\G_*$, say $T$. Thus $T^*$ is the only operator that 
satisfies the conclusion.  
\end{proof}

If $\m_\varphi$ is completely bounded on $\GG$, we say that $\m_\varphi$ is a {\em radial 
multiplier} on $\G$.

\begin{theo}\label{main}
Let $\varphi: \N \to \C$ be any function and $H_\R$ an infinite dimensional real Hilbert
space with a one-parameter group $(U_t)$ of orthogonal transformations. Then $\varphi$ defines a completely bounded radial multiplier
on $\G$ if and only if the radial functional $\gamma$ on $\mathcal{T}$
associated to $\varphi$ is bounded. Moreover
$$\| \m_\varphi\|_{\cb}= \| \gamma \|_{\mathcal{T}^*}.$$
\end{theo}
Thanks to Proposition \ref{calcul}, we have an explicit formula for 
$\| \gamma \|_{\mathcal{T}^*}$.

\begin{proof}[Proof of the upper bound]
We assume that $\varphi$ gives a bounded functional $\gamma$ on $\mathcal{T}$.
By the universal property of $\Ch$, there is a $\ast$-homomorphism
$$\pi : \begin{array}{ccc}\Ch &\to  & \Ch \otimes_{\min} \mathcal{T}\\
\ell(\xi) &\mapsto & \ell(\xi)\otimes S
\end{array}. $$
So the map $\m_\varphi =(\Id \otimes \gamma) \pi : \Ch \to \Ch$ is completely bounded on
$\Ch$ with norm $||\gamma||_{\mathcal{T}^*}$. We have, for all $n \in \N$ and all $e_{k} \in K_\R + i K_\R$:
$$\m_\varphi(\ell(e_{1}) \cdots \ell(e_{k}) \ell(\overline e_{{k+1}})^* \cdots \ell(\overline
e_{n})^*)=\varphi(n) \ell(e_{1}) \cdots \ell(e_{k}) \ell(\overline
e_{{k+1}})^* \cdots \ell(\overline e_{n})^*$$ 
Recall that the Wick formula (Lemma \ref{Wick}) says
$$W(e_1\otimes \cdots\otimes e_n)=\sum_{k=0}^n \ell(e_{1}) \cdots \ell(e_{k}) \ell(\overline
e_{{k+1}})^* \cdots \ell(\overline e_{n})^*.$$ 
Thus we derive that $\m_\varphi(W(e_1\otimes \cdots \otimes e_n))= \varphi(n)W(e_1\otimes \cdots \otimes e_n)$. So $\m_\varphi$ is bounded on $\GG$ and is a radial multiplier.
\end{proof}

To check the necessity of the condition, the idea is similar to
\cite{HSS} or \cite{Wys}. We find a shift algebra on which $\m_\varphi$
acts. We start by taking an orthonormal system $(e_{i})_{i\geq 1}$ in $K_\R + i K_\R$ such that $\langle \overline e_i,
\overline e_j\rangle = \langle \overline e_i,
e_j\rangle=0$ for all $i\neq j$ and $\|\overline e_i\|\leq 1$ (this is
possible by the Gram-Schmidt algorithm). Consider the following element 
for $n\geq 1$:
$$S_n=\frac 1 {\sqrt n} \sum_{i=1}^n \ell(e_i)\otimes W(e_i) \in \mathcal{T}(H) \otimes \mathbf{B}(\mathcal{F}(H)).$$

\begin{lem}\label{cas00}
For all $n\geq 1$, $$||S_n^*S_n-1||_\infty \leq \frac  3 {\sqrt n}.$$
\end{lem}

\begin{proof}
We have
\begin{eqnarray*}
W(e_i)^*W(e_i) & = & (\ell(e_i)^* + \ell(\overline e_i)) (\ell(e_i) + \ell(\overline e_i)^*) \\
& = & 1 + \ell(\overline e_i) \ell(e_i) + \ell(\overline e_i) \ell(\overline e_i)^* + \ell(e_i)^*\ell(\overline e_i)^*\\
& = & 1+ W(\overline e_i \otimes e_i).
\end{eqnarray*}
It follows that 
\begin{eqnarray*}
S_n^*S_n &= & \frac 1 n \sum_{i,j=1}^n \ell(e_i)^* \ell(e_j)\otimes W(e_i)^*W(e_j)\\
& =& 1\otimes 1 + \frac 1 n \sum_{i=1}^n 1 \otimes (\ell(\overline e_i) \ell(e_i) + \ell(\overline e_i) \ell(\overline e_i)^* + \ell(e_i)^*\ell(\overline e_i)^*).
\end{eqnarray*}
Lemma \ref{majf} yields
$$\left\|\frac 1 n\sum_{i=1}^n 1 \otimes W(\overline e_i \otimes e_i)\right\|_\infty \leq \frac 3{\sqrt n},$$ so that we get the estimate.
\end{proof}
For convenience, we will use a standard multi-index notation, we write $\underline i$ for $(i_1,\dots,i_n)\in \N^n$ and $|\underline i|=n$. 
For $\alpha,\beta\geq 0$, set
\begin{eqnarray*}
e_{\underline{i}}^{\alpha, \beta} & = & e_{i_1}\otimes \cdots \otimes e_{i_\alpha} \otimes \overline e_{i_{\alpha+1}} \cdots \otimes \overline e_{i_{\alpha+\beta}} \\
V_{\alpha,\beta}^n & = & n^{-\frac{\alpha+\beta}{2}} \sum_{i_1, \dots,i_{\alpha+\beta}=1}^n
\ell(e_{i_1}) \cdots \ell(e_{i_{\alpha}}) \ell( e_{i_{\alpha+1}})^*\cdots \ell(e_{i_{\alpha+\beta}})^*
\otimes W(e_{\underline{i}}^{\alpha, \beta}),
\end{eqnarray*}
if $\alpha + \beta>0$ and $V_{0,0}^n=1\otimes 1$.

\begin{lem}\label{cas11}
For all $\alpha,\beta\geq 0$, 
$$S_n^{\alpha}S_n^{*\beta}-V_{\alpha,\beta}^n= O\left(\frac 1{\sqrt n}\right).$$
\end{lem}

\begin{proof}
We do it by induction on $\alpha + \beta$. When $\alpha+\beta\leq 1$, there
is equality. Assume this holds for $(\alpha,\beta)$, we prove it for $(\alpha+1,\beta)$.
 First,  $S_nV_{\alpha,\beta}^n$ is equal to 
$$n^{-\frac{\alpha+\beta+1}{2}}\sum_{i_0, \dots, i_{\alpha+\beta} = 1}^n
\ell(e_{i_0}) \cdots \ell(e_{i_{\alpha}})
\ell(e_{i_{\alpha+1}})^* \cdots \ell(e_{i_{\alpha+\beta}})^*
\otimes W(e_{i_0})W(e_{\underline{i}}^\alpha).$$
Recall the identity 
$$W(h)W(h_1\otimes \cdots )=W(h \otimes h_1 \otimes \cdots) + \langle \overline h,h_1\rangle W(h_2 \otimes \cdots)$$ 
used in the proof of the Wick formula. Therefore 
$$S_nV_{\alpha,\beta}^n = V_{\alpha+1,\beta}^n + \left(\frac 1 n \sum_{i=1}^n \langle \overline e_i,e_i^{(*)}\rangle \ell(e_i)\ell(e_i)^{(*)}\otimes 1\right) V_{\tilde \alpha, \tilde \beta}^n$$
where $(*)=1$, $\tilde \alpha=\alpha-1$, $\tilde \beta=\beta$ and $e_i^{(*)}=e_i$ if $\alpha>0$,
and $\ell(e_i)^{(*)} = \ell(\overline e_i)^*$, $\tilde \alpha=0$, $\tilde \beta=\beta-1$ and $e_i^{(*)}=\overline e_i$ if $\alpha=0$.
We have by Lemma \ref{majf} $\frac 1 n\sum_{i=1}^n \langle \overline e_i,e_i^{(*)}\rangle \ell(e_i) \ell(\overline e_i)^{*}=O\left(\frac 1 n\right)$ and 
$\frac 1 n\sum_{i=1}^n \langle \overline e_i,e_i^{(*)}\rangle \ell(e_i) \ell(e_i)=O\left(\frac 1 {\sqrt{n}}\right)$. This yields
$S_nV_{\alpha,\beta}^n - V^n_{\alpha+1,\beta}=O\left(\frac 1 {\sqrt{n}}\right)$.
According to Lemma \ref{cas00} and the induction hypothesis, $S_n$ and
then $V_{a,b}^n$ for $a+b\leq \alpha+\beta$ are uniformly bounded in
 $n$. Consequently,
$$S_n^{\alpha+1}S_n^{*\beta}-V_{\alpha+1,\beta}^n= S_n\left(S_n^{\alpha}S_n^{*\beta}-V_{\alpha,\beta}^n\right) + O\left(\frac 1 {\sqrt{n}}\right)= O\left(\frac 1 {\sqrt{n}}\right).$$
The other case $(\alpha,\beta+1)$ is obtained by taking adjoints.
\end{proof}

\begin{proof}[Proof of the lower bound]
Assume $\m_\varphi$ is a completely bounded multiplier on the free Araki-Woods factor $\G$.
Let $\mathfrak U$ be a nontrivial ultrafilter on $\N$. Set $\mathcal{B} = \mathcal{T}(H) \otimes \mathbf{B}(\mathcal{F}(H))$ so that $S_n \in \mathcal{B}$. Consider the $C^*$-algebra 
$\mathcal{A}=\prod_{\mathfrak U} \mathcal{B}$, and $T$ the ultrapower of $\Id \otimes \m_\varphi$.
The element $S=(S_n)\in \mathcal{A}$ satisfies $S^*S=1$ by Lemma
\ref{cas00}.  As $(\Id \otimes
\m_\varphi)(V_{\alpha,\beta}^n)=\varphi(\alpha + \beta)V_{\alpha,\beta}^n$, we
get by Lemma \ref{cas11}, $T(S^\alpha S^{*\beta})=\varphi(\alpha + \beta)
S^\alpha S^{*\beta}$. Taking a particular non constant $\varphi$ (that
does exist), this shows that $S$ is non unitary and $S$ is a shift. Thus, $T$
leaves $\mathcal{T}=\langle S\rangle$ invariant. By composing it with the trivial character $\omega$ of 
$\mathcal{T}$ ($\omega(S^\alpha S^{*\beta})=1)$, we obtain that $\gamma = \omega T$ is 
a bounded functional on $\mathcal{T}$ with $||\gamma||_{\mathcal{T}^*} \leq ||\m_\varphi||_{\cb}$.
\end{proof}

A linear map between $C^*$-algebras $\mathcal A$ and $\mathcal B$ is
decomposable if it is a linear combination of completely positive maps
from $\mathcal A$ to $\mathcal B$. Any functional can be decomposed
into sums of states, so we have:

\begin{cor}
Any radial multiplier on $\GG$ is decomposable from $\GG$ into $\Ch$.
\end{cor}

More generally, a function $\varphi : \N \to \C$ defines a radial
multiplier on $\Ch$ if the map $T_\varphi$ given by
$$T_\varphi(\ell(e_{1}) \cdots \ell(e_{k}) \ell(e_{{k+1}})^* \cdots \ell(e_{n})^*)=\varphi(n) \ell(e_{1}) \cdots \ell(e_{k}) \ell(e_{{k+1}})^* \cdots \ell(e_{n})^*$$ 
extends to a completely bounded map on $\Ch$. The above proof actually gives the following

\begin{cor}
For $\varphi : \N \to \C$, we have $||\m_\varphi||_{\cb}=||T_\varphi||_{\cb}=||T_\varphi||$.
\end{cor}

\begin{rem}
The situation is very similar to the one of Herz-Schur and Schur
multipliers.  We have an extension of $\m_\varphi$ to a larger algebra
that remains a multiplier. Moreover its bounded and completely bounded
norms coincide.  Of course, by Stinespring's theorem, there is always
an extension of $\m_\varphi$ to a map from $\mathcal T(H)$ to
$\mathbf{B}(\mathcal F(H))$.  The point is that we cannot ensure it to be
$T_\varphi$, whereas for Herz-Schur and Schur multipliers this fact is
easy (this is related to Gilbert's argument). It is straightforward to
deduce the main result of this section from this corollary, but
unfortunately we have no direct way to prove it.
\end{rem}

\begin{rem}
We have to deal with completely bounded norms. Take a trivial $(U_t)$ and
$\varphi(n) = \delta_{n,1}$. We have $\|\m_\varphi\|_{\cb}=2$
by the above theorem. But because of the
invariance under orthogonal transformations of semicircular random
variables, it is standard to check that $\|\m_{\varphi}\|=\frac
{16}{3\pi}=\|s\|_1\|s\|_\infty$, where $s$ is a normalized
semicircular random variable.
\end{rem}

Taking $\delta_{\leq d}(n)=\delta_{n \leq d}$, the corresponding multiplier $P_d$ on $\Gamma(H_\R, U_t)$ is called the {\em projection} onto words
of length less than $d$. Thanks to Proposition \ref{calcul}, we get:

\begin{cor}
For any orthogonal group $(U_t)$ on an infinite dimensional real Hilbert space $H_\R$, $$\|P_d\|_{\cb(\Gamma(H_\R, U_t))} \mathop{\sim}\limits_{d\to \infty} \frac 4\pi d.$$
\end{cor}

\begin{proof}
We apply Theorem \ref{main} and Proposition \ref{calcul} to this
particular radial function. It is clear that $c_1=c_2=0$ in \ref{calcul}. It remains to estimate the trace
norm of $B=\sum_{i=0}^d e_{i,d-i} +\sum_{i=0}^{d-1} e_{i,d-1-i}$. 
To do so, $B+e_{d,d}$ is unitarily equivalent to a circulant matrix of size $d+1$, 
$\Id_{d+1}+J_{d+1}$ where $J_{d+1}=\sum_{i=0}^d e_{i,i+1}$. The singular values of $B$
are exactly $1+e^{\frac {2i\pi k}{d+1}}$, for $k=0,\dots,d$. We get that $\|B\|_1/d$ tends to $\int_0^1 |1+e^{2i\pi t}| {\rm d}t=\frac 4 \pi$.  
 \end{proof}

\begin{rem} The upper bound can be established directly (with a worse constant) using the argument of Chapter 3 in \cite{RicardXu}. We point out that   
the lower bound in \ref{main} remains true for the $q$-deformed algebras. 
\end{rem}

\begin{cor}\label{finrad}
For any orthogonal group $(U_t)$ on an infinite dimensional real Hilbert space
$H_\R$, there are finitely supported functions $\varphi_n : \N \to \R$ such that
$\lim_{n} \|\m_{\varphi_n}\|_{\cb} = 1$ and $\lim_n\varphi_n(k) = 1$ for all $k\geq 0$.
\end{cor}

\begin{proof}
This is an argument due to Haagerup \cite{haa} (see also
\cite{RicardXu}). Using Corollary \ref{ucp}
below or Theorem \ref{main}, the contraction $H \ni \xi \mapsto e^{-t} \xi \in H$ gives rise
to a unital completely positive multiplier $\m_{\psi_t}$ on $\G$ (for
$t \geq 0$) where $\psi_t(k)=e^{-kt}$. Since
$$\psi_t = \sum_d e^{-dt}\delta_d = \sum_d e^{-dt}(\delta_{\leq d} - \delta_{\leq d - 1}),$$ the polynomial estimate gives that 
$$\limsup_{d\to \infty} \|\m_{\psi_t}(1-P_d)\|_{\cb}  \leq  \limsup_{d \to \infty} \sum_{k \geq d} e^{-kt}\|P_{k + 1} - P_k\|_{\cb} = 0.$$
For every $n \geq 1$, choose $d_n$ large enough so that $\|\m_{\psi_{1/n}}(1 - P_{d_n})\|_{\cb} \leq 1/n$.  The net of the form $\varphi_n = \psi_{1/n} \delta_{\leq d_n}$ satisfies the conclusion of the corollary.
\end{proof}

\subsection{Approximation properties}
We follow a very typical approach. We first establish a second
quantization procedure on free Araki-Woods von Neumann algebras, which generalizes \cite{{shlya97}, {voiculescu85}}. Then,
to get the approximation property, we just need to cut them
with some radial multipliers to get finite rank maps. Let $H$ and $K$ be Hilbert spaces and let $T: H \to K$ be a contraction. We will denote the corresponding first quantization $\mathcal{F}(H) \to \mathcal{F}(K)$ by
$$\tilde{\Gamma}(T) = 1 \oplus \bigoplus_{n \geq 1} T^{\otimes n}.$$

\begin{theo}\label{functor}
Let $H$ and $K$ be Hilbert spaces and $T: H \to K$ be a
contraction. Then there is a unique unital completely positive map
$\Gamma(T) : \Ch \to \mathcal{T}(K)$ such that
\begin{eqnarray*}
& & \Gamma(T)(\ell(h_1) \cdots \ell(h_k) \ell(h_{k+1})^* \cdots \ell(h_{n})^*) = \\
& & \ell(T(h_1)) \cdots \ell(T(h_k)) \ell(T(h_{k+1}))^* \cdots \ell(T(h_{n}))^*
\end{eqnarray*}
for all $h_i \in H$.
\end{theo}

\begin{proof}
This is again a consequence of the universal property of $\Ch$. It is clear that if 
$\Gamma(T)$ and $\Gamma(S)$ exist then $\Gamma(ST)=\Gamma(S)\Gamma(T)$. 
So by the general form of a contraction, one just needs to prove the result
when $T$ is either an inclusion from $H$ to $K$, or a unitary on $H$,  
or an orthogonal projection from $H$ to $K$.

If $T$ is an inclusion, this is just the universal property of $\Ch$
(note that $\Gamma(T)$ is an injective $*$-representation). We
emphasize that if $H\subset K$ and $h\in H$, then $\ell(h)$ has {\em a
priori} two different meanings as a creation operator on $\mathcal
F(H)$ or $\mathcal F(K)$. The universal property tells us that there is no 
difference at the $C^*$-level.

 If $T$ is a unitary, this is also the universal
property, but in this case $\Gamma(T)$ is nothing but the restriction
of the conjugation by the unitary $\tilde{\Gamma}(T)$ on the full Fock
space $\mathcal{F}(H)$.

If $T$ is an orthogonal projection from $H$ to $K$, we write $j : K\to
H$ for the inclusion. The first quantization $\tilde \Gamma(j)=\iota$
is also an inclusion of $\mathcal{F}(K)$ into
$\mathcal{F}(H)$, the orthogonal projection 
$\iota^*$ is exactly $\tilde \Gamma(T)$. To avoid any confusion, for $k \in K$, write $\ell_K(k) : \mathcal F(K) \to \mathcal F(K)$ for the creation operator on $\mathcal F(K)$ and $\ell_H(k) : \mathcal F(H) \to \mathcal F(H)$ for the creation operator on $\mathcal F(H)$. For
$h\in H$ and $k\in K$, we have $\ell_H(h)^*k = \langle h, k \rangle
\Omega = \langle T(h), k\rangle \Omega= \ell_H(T(h))^*k = \ell_K(T(h))^*k$. This yields
\begin{eqnarray*}
& & \iota^* \ell_H(h_1)\cdots \ell_H(h_k)\ell_H(h_{k+1})^* \cdots \ell_H(h_{n})^*\iota = \\
& & \ell_K(T(h_1)) \cdots \ell_K(T(h_k)) \ell_K(T(h_{k+1}))^*\cdots \ell_K(T(h_{n}))^*. \nonumber
\end{eqnarray*}
Hence $\Gamma(T)(x) = \iota^* x \iota$, for all $x \in \mathcal T(H)$. It is then clear that $\Gamma(T) : \mathcal T(H) \to \mathcal T(K)$ is completely positive.
\end{proof}

We come back to the free Araki-Woods algebras with the notation of the
previous sections. The second quantization is usually stated for maps such that $AT=TA$
which is a somewhat strong assumption \cite{shlya97}. This was the main obstacle to
prove approximation properties for general free Araki-Woods algebras as there
can be no finite rank $T$ satisfying that condition.

\begin{cor}\label{ucp}
Let $T: H \to H$ be a contraction so that $I T I=T$. Then
$\Gamma(T)$ leaves $\GG$ invariant and $\Gamma(T)$ extends to a normal completely positive map on $\G$ so that
$$\Gamma(T)W(\xi)=W(\tilde{\Gamma}(T)\xi), \forall \xi \in \G \Omega.$$
\end{cor}

\begin{proof}
If $I T I=T$, this implies that for all $\xi \in K_\R + i K_\R$, we have
$T(\overline\xi)=\overline{T(\xi)}$. So by the Wick formula for $e_i$ in $K_\R + i K_\R$, we have
\begin{eqnarray*}
\Gamma(T)W(e_1\otimes \cdots \otimes e_n) & = & \sum_{k=0}^n  
\ell(T(e_{1})) \cdots \ell(T(e_{k})) \ell(T(\overline e_{{k+1}}))^*
\cdots \ell(T(\overline e_{n}))^* \\ 
& = & \sum_{k=0}^n  
\ell(T(e_{1})) \cdots \ell({T(e_{k}))} \ell(\overline {T( e_{{k+1}})})^*
\cdots \ell(\overline{T(e_{n})})^* \\ 
& = & W(T(e_1)\otimes \cdots \otimes T(e_n)).
\end{eqnarray*}
As the set of such elements is linearly dense in $\GG$, we get that
$\GG$ is stable by $\Gamma(T)$. The normal extension is done as in Lemma \ref{norext}.
\end{proof}

\begin{prop}
There is a net of finite rank contractions $(T_k)_{k}$ converging to the 
identity on $H$ pointwise, such that $T_k = I T_k I$, for every $k$. 
\end{prop}

\begin{proof}
Let $(\mathbf{1}_{[\lambda,\infty]}(A))_{\lambda\geq 0}$ be the
spectral projections of $A$.
Since $I A I=A^{-1}$, we get
$$I \mathbf{1}_{[\lambda,\infty[}(A)(H)=\mathbf{1}_{[0,1/\lambda]}(A)(H).$$ 
Recall that $I = JA^{-1/2}$ is the polar decomposition of $I$. We also have $JAJ=A^{-1}$ and $J$ is an anti-unitary that sends $\mathbf{1}_{[\lambda,\beta]}(A)(H)$ to
$\mathbf{1}_{[1/\beta,1/\lambda]}(A)(H)$.

Fix $\lambda > 1$ and $0 < \delta < 1$. Take a 
subspace $E$ in $\mathbf{1}_{[\lambda,\lambda+\delta]}(A)(H)$ and
denote by $P$ the orthogonal projection onto $E$. We show that $I P
I$ is almost the orthogonal projection $JPJ$. Indeed, we have $$I P
I = JA^{-1/2}\mathbf{1}_{[\lambda,\lambda+\delta]}(A)P
\mathbf{1}_{[\lambda,\lambda+\delta]}(A)JA^{-1/2}\mathbf{1}_{[\frac
1{\lambda+\delta},\frac 1 \lambda]}(A).$$ 
Moreover
\begin{eqnarray*}
\left\|A^{-1/2}\mathbf{1}_{[\lambda,\lambda+\delta]}(A)-\frac 1 {\sqrt
\lambda}\mathbf{1}_{[\lambda,\lambda+\delta]}(A)\right\|_\infty & \leq & \frac{\delta}{2\sqrt{\lambda}^3} \\
\left\|A^{-1/2}\mathbf{1}_{[\frac 1{\lambda+\delta},\frac 1 \lambda]}(A)-\sqrt
\lambda \mathbf{1}_{[\frac 1{\lambda+\delta},\frac 1 \lambda]}(A)\right\|_\infty & \leq & \frac{\delta}{2\sqrt{\lambda}}.
\end{eqnarray*}
The triangle inequality gives
$$\| I P I -J P J\|_\infty \leq \frac \delta {2\lambda}+ \frac \delta
{2\lambda}+\frac {\delta^2} {4\lambda^2} \leq \frac{2 \delta}{\lambda}.$$ 
Summarizing, for any finite dimensional subspace 
$E\subset \mathbf{1}_{[\lambda,\lambda+\delta]}(A)(H)$ and corresponding
projections $P_E$, $T_E = \frac{1}{1 + \frac{2\delta}{\lambda}}(P_E \oplus I P_E I)$ is a finite rank contraction that satisfies $I T_E I = T_E$ and
$$\|T_E - (P_E \oplus J P_E J)\|_\infty < \frac{4 \delta}{\lambda}.$$
Observe that for operators $S$ and $T$ which have orthogonal left and right supports, we denote the sum $S + T$ by $S \oplus T$.

Take $F$ a finite dimensional subspace of $H$ and fix
$\varepsilon > 0$. Then there exists $n \in \N$ such that for all $f\in F$, we have
$\|\mathbf{1}_{[e^{-n},e^n]}(A)f-f\| \leq (\varepsilon/3) \|f\|$. Set
$\lambda_k=e^{nk/N}$, for $1 \leq k \leq N$ for some large $N$ chosen later.  Let $P_k$
be the orthogonal projection onto $\mathbf{1}_{[\frac 1 {\lambda_{k+1}},
\frac 1 \lambda_k]}(A)(H)\oplus \mathbf{1}_{[\lambda_{k},
\lambda_{k+1}]}(A)(H)$ for $k\geq 1$, and $P_0$ be the projection onto the
eigenspace of $A$ for $1$. Observe that $\frac{\lambda_{k + 1} - \lambda_k}{\lambda_k} = e^{n/N} - 1$.

By the above construction, for each $1 \leq k \leq N$, we can find a finite
rank contraction $T_k$ on $P_k(H)$ such that $I T_k I=T_k$ and for
every $f\in F$, 
$$\|T_k(P_k f) - P_k f \| \leq 4 (e^{n/N}-1) \|P_k f\|.$$ For $k=0$,
as $I$ is an anti-unitary on $P_0(H)$, we take $T_0$ the orthogonal
projection onto $P_0(F)+IP_0(F)$, it satisfies the above properties
with $k=0$.

Set $T=\bigoplus_{k = 0}^N T_k$, which is a finite rank contraction as the $T_k$'s act on
orthogonal subspaces. Moreover $I T I=T$ and for all $f\in F$,
gathering the estimates 
$$\|T(f)-f\|\leq 4 (e^{n/N}-1) \|f\| + \|\mathbf{1}_{[e^{- n/N},e^{n/N}]\backslash\{1\}}(A)f\| +  (2 \varepsilon/3)\|f\|.$$ 
Letting $N \to \infty$, this upper bound can be made smaller than $\varepsilon \|f\|$. So we get the conclusion with a net index by finite dimensional 
subspace of $H$ and $\varepsilon>0$.
\end{proof}

\begin{theo}[Theorem A]
The von Neumann algebra $\G$ has the complete metric approximation property.
\end{theo}

\begin{proof}
Using the contractions of the previous Proposition, the net
$(\Gamma(T_k))_k$ is made of unital completely positive maps which tend
pointwise to the identity. Let $(\m_{\varphi_n})$ be the multipliers from Corollary \ref{finrad}. Since 
$$(\m_{\varphi_n} \circ \Gamma(T_k))(W(e_{\underline{i}})) = \varphi_n(|\underline {i}|)W(\tilde{\Gamma}(T_k)e_{\underline{i}}),$$
the net $(\m_{\varphi_n} \circ \Gamma(T_k))_{n,k}$ are normal finite rank
completely bounded maps which satisfy:
\begin{itemize}
\item $\lim_n \lim_k (\m_{\varphi_n} \circ \Gamma(T_k)) = \Id$ pointwise $\ast$-strongly and
\item $\lim_n \lim_k \|\m_{\varphi_n} \circ \Gamma(T_k)\|_{\cb} = 1$.
\end{itemize} 
The proof is complete.
\end{proof}

There is another approximation property that turns out to be useful. A von Neumann algebra $M$ satisfies the {\it Haagerup property} if there exists a net 
$(u_i)_{i\in I}$ of normal completely positive maps from $M$ to $M$ such that
\begin{enumerate}
\item for all $x\in M$, $u_i(x)\to x$ $\sigma$-weakly.
\item for all $\xi\in L^2(M)$ and $i\in I$ the map $x\mapsto u_i(x)\xi$
is compact from $M$ to $L^2(M)$.
\end{enumerate}

\begin{theo}
The von Neumann algebra $\G$ has the Haagerup property.
\end{theo}

\begin{proof}
This is just a variation. As above, with the finite
rank maps of the previous Proposition, it is easy to check that
$(\Gamma(e^{-t}T_k))_{t>0,k\in \N}$ is a net of unital completely
positive maps that tends to the identity pointwise with respect to the $\sigma$-weak topology. It remains only to check the second point.

We use the notation of the proof of Corollary \ref{finrad}. We have
$\Gamma(e^{-t})=\m_{\psi_t}$ and 
$$\lim_{d\to \infty} \|\m_{\psi_t}(1-P_d)\|_{\cb}=0.$$ So
$\Gamma(e^{-t}T_k)=\m_{\psi_t}(1-P_d)\Gamma(T_k)+P_d\Gamma(e^{-t}T_k)$,
as $P_d\Gamma(e^{-t}T_k)$ is finite rank, $\Gamma(e^{-t}T_k)$ is a
limit in norm of finite rank operators so is compact from $\G$ to
$\G$. In particular, its composition with the evaluation on a vector
$\xi\in L^2(\G)$ is also compact.
\end{proof}

\section{Malleable deformation on free Araki-Woods factors}\label{malleable}

\subsection{The free malleable deformation}

We first introduce some notation we will be using throughout this section. Let $H_\R$ be a separable real Hilbert space ($\dim H_\R \geq 2$) together with $(U_t)$ an orthogonal representation of $\R$ on $H_\R$. We set: 
\begin{itemize}
\item $\mathcal{M} = \Gamma(H_\R, U_t)''$ the free Araki-Woods factor associated with $(H_\R, U_t)$. Denote by $\chi$ the free quasi-free state and by $\sigma$ the modular group of the state $\chi$. 
 
\item $M = \mathcal{M} \rtimes_\sigma \R$ is the continuous core of $\mathcal{M}$ and $\Tr$ is the semifinite trace associated with the state $\chi$. 

\item Likewise $\widetilde{\mathcal{M}} = \Gamma(H_\R \oplus H_\R, U_t \oplus U_t)''$, $\widetilde{\chi}$ is the corresponding free quasi-free state and $\widetilde{\sigma}$ is the modular group of $\widetilde{\chi}$.

\item $\widetilde{M} = \widetilde{\mathcal{M}} \rtimes_{\widetilde{\sigma}} \R$ is the continuous core of $\widetilde{\mathcal{M}}$ and $\widetilde{\Tr}$ is the semifinite trace associated with $\widetilde{\chi}$.
\end{itemize}
It follows from \cite{shlya97} that
\begin{equation*}
 \widetilde{\mathcal{M}} \cong \mathcal{M} \ast \mathcal{M}.
\end{equation*}
In the latter free product, we shall write $\mathcal{M}_1$ for the first copy of $\mathcal{M}$ and $\mathcal{M}_2$ for the second copy of $\mathcal{M}$. We regard $\mathcal{M} \subset \widetilde{\mathcal{M}}$ via the identification of $\mathcal{M}$ with $\mathcal{M}_1$.

Denote by $(\lambda_t)$ the unitaries in $L(\R)$ that implement the modular action $\sigma$ on $\mathcal{M}$ (resp. $\widetilde{\sigma}$ on $\widetilde{\mathcal{M}}$). Define the following faithful normal conditional expectations:
\begin{itemize}
\item $E : M \to L(\R)$ such that $E(x \lambda_t) = \chi(x) \lambda_t$, for every $x \in \mathcal{M}$ and $t \in \R$;
\item $\widetilde{E} : \widetilde{M} \to L(\R)$ such that $\widetilde{E}(x \lambda_t) = \widetilde{\chi}(x) \lambda_t$, for every $x \in \widetilde{\mathcal{M}}$ and $t \in \R$.
\end{itemize}
Then
\begin{equation*}
\left( \widetilde{M}, \widetilde{E} \right) \cong (M, E) \ast_{L(\R)} (M, E).
\end{equation*}
Likewise, in the latter amalgamated free product, we shall write $M_1$ for the first copy of $M$ and $M_2$ for the second copy of $M$. We regard $M \subset \widetilde{M}$ via the identification of $M$ with $M_1$. Notice that the conditional expectation $E$ (resp. $\widetilde{E}$) preserves the canonical semifinite trace $\Tr$ (resp. $\widetilde{\Tr}$) associated with the state $\chi$ (resp. $\widetilde{\chi}$) (see \cite{ueda}).

Consider the following orthogonal representation of $\R$ on $H_\R \oplus H_\R$:
\begin{equation*}
V_s = \begin{pmatrix}
\cos(\frac{\pi}{2}s) & -\sin(\frac{\pi}{2}s) \\
\sin(\frac{\pi}{2}s) & \cos(\frac{\pi}{2}s)
\end{pmatrix}, \forall s \in \R.
\end{equation*}
Let $(\alpha_s)$ be the natural action on $\left( \widetilde{\mathcal{M}}, \widetilde{\chi} \right)$ associated with $(V_s)$: 
$$\alpha_s = \Gamma(V_s), \forall s \in \R.$$ 
In particular, we have
\begin{equation*}
\alpha_s(W\begin{pmatrix}
\xi \\ \eta \end{pmatrix}) =  W(V_s\begin{pmatrix}
\xi \\ \eta \end{pmatrix}), \forall s \in \R, \forall \xi, \eta \in H_\R,
\end{equation*}
and the action $(\alpha_s)$ is $\widetilde{\chi}$-preserving. We can easily see that the representation $(V_s)$ commutes with the representation $(U_t \oplus U_t)$. Consequently,  $(\alpha_s)$ commutes with modular action $\widetilde{\sigma}$. Moreover, $\alpha_1(x \ast 1) = 1 \ast x$, for every $x \in \mathcal{M}$. At last,  consider the automorphism $\beta$ defined on $\left( \widetilde{\mathcal{M}}, \widetilde{\chi} \right)$ by:
\begin{equation*}
\beta(W\begin{pmatrix}
\xi \\ \eta \end{pmatrix}) = W\begin{pmatrix}
\xi \\ -\eta \end{pmatrix}, \forall \xi, \eta \in H_\R. 
\end{equation*}
It is straightforward to check that $\beta$ commutes with the modular action $\widetilde{\sigma}$, $\beta^2 = \Id$, $\beta_{|\mathcal{M}} = \Id_{\mathcal{M}}$ and $\beta\alpha_{s} = \alpha_{-s}\beta$, $\forall s \in \R$. Since $(\alpha_s)$ and $\beta$ commute with the modular action $\widetilde{\sigma}$, one may extend $(\alpha_s)$ and $\beta$ to $\widetilde{M}$ by ${\alpha_s}_{|L(\R)} = \Id_{L(\R)}$, for every $s \in \R$ and $\beta_{|L(\R)} = \Id_{L(\R)}$. Moreover $(\alpha_s, \beta)$ preserves the semifinite trace $\widetilde{\Tr}$. We summarize what we have done so far:

\begin{prop}
The $\widetilde{\Tr}$-preserving deformation $(\alpha_s, \beta)$ defined on $\widetilde{M} = M \ast_{L(\R)} M$ is \emph{s-malleable}:
\begin{enumerate}
\item ${\alpha_s}_{|L(\R)} = \Id_{L(\R)}$, for every $s \in \R$ and $\alpha_1(x \ast_{L(\R)} 1) = 1 \ast_{L(\R)} x$, for every $x \in M$. 
\item $\beta^2 = \Id$ and $\beta_{|M} = \Id_{M}$. 
\item $\beta \alpha_s = \alpha_{-s} \beta$, for every $s \in \R$.
\end{enumerate}
\end{prop}

Denote by $E_M : \widetilde{M} \to M$ the canonical trace-preserving conditional expectation. Since $\widetilde{\Tr}_{|M} = \Tr$, we will simply denote by $\Tr$ the semifinite trace on $\widetilde{M}$. 
Recall that the s-malleable deformation $(\alpha_s, \beta)$ automatically features a certain {\it transversality property}. 

\begin{prop}[Popa, \cite{popasup}]\label{transversality}
We have the following:
\begin{equation}\label{trans}
\| x - \alpha_{2s}(x) \|_{2, \Tr} \leq 2 \| \alpha_s(x) - (E_{M}\circ \alpha_s)(x) \|_{2, \Tr}, \; \forall x \in L^2(M, \Tr), \forall s > 0.
\end{equation}
\end{prop}

\subsection{Locating subalgebras inside the core}

The following theorem is in some ways reminiscent of a result by Ioana, Peterson and Popa, namely \cite[Theorem 4.3]{ipp} (see also \cite[Theorem 4.2]{houdayer4} and \cite[Theorem 3.4]{houdayer6}). 

\begin{theo}\label{uniform}
Let $\mathcal{M} = \Gamma(H_\R, U_t)''$ and $M = \mathcal{M} \rtimes_\sigma \R$ be as above. Let $p \in L(\R) \subset M$ be a nonzero projection such that $\Tr(p) < \infty$. Let $P \subset pMp$ be a von Neumann subalgebra such that the deformation $(\alpha_t)$ converges uniformly in $\| \cdot \|_{2, \Tr}$ on $\mathcal{U}(P)$. Then $P \preceq_M L(\R)$.
\end{theo}

\begin{proof}
Let $p \in L(\R)$ be a nonzero projection such that $\Tr(p) < \infty$. Let $P \subset pMp$ be a von Neumann subalgebra such that $(\alpha_t)$ converges uniformly in $\|\cdot\|_{2, \Tr} $ on $\mathcal{U}(P)$. We keep the notation introduced previously and regard $M \subset \widetilde{M} = M_1 \ast_{L(\R)} M_2$ via the identification of $M$ with $M_1$. Recall that ${\alpha_s}_{|L(\R)} = \Id_{L(\R)}$, for every $s \in \R$. In particular, $\alpha_s(p) = p$, for every $s \in \R$.

{\bf Step (1) : Using the uniform convergence on $\mathcal{U}(P)$ to find $t > 0$ and a nonzero intertwiner $v$ between $\Id$ and $\alpha_t$.}

The first step uses a standard functional analysis trick. Let $\varepsilon = \frac{1}{2}\left\| p \right\|_{2, \Tr}$. We know that there exists $s = 1/2^k$ such that $\forall u \in \mathcal{U}(P)$, 
\begin{equation*}
\| u - \alpha_s(u) \|_{2, \Tr} \leq \frac{1}{2}\| p \|_{2, \Tr},
\end{equation*}
Thus, $\forall u \in \mathcal{U}(P)$, we have
\begin{eqnarray*}
\| u^*\alpha_s(u) - p \|_{2, \Tr} & = & \| u^*(\alpha_s(u) - u) \|_{2, \Tr} \\
& \leq & \| u - \alpha_s(u) \|_{2, \Tr} \\
& \leq & \frac{1}{2} \| p \|_{2, \Tr}.
\end{eqnarray*}
Denote by $\mathcal{C} = \overline{\co}^w \{u^*\alpha_s(u) : u \in \mathcal{U}(P)\} \subset p L^2(\widetilde{M})p$ the ultraweak closure of the convex hull of all $u^*\alpha_s(u)$, where $u \in \mathcal{U}(P)$. Denote by $a$ the unique element in $\mathcal{C}$ of minimal $\| \cdot \|_{2, \Tr}$-norm. Since $\| a - p \|_{2, \Tr} \leq 1/2 \| p \|_{2, \Tr}$, necessarily $a \neq 0$. Fix $u \in \mathcal{U}(P)$. Since $u^* a \alpha_s(u) \in \mathcal{C}$ and $\| u^* a \alpha_s(u) \|_{2, \Tr} = \| a \|_{2, \Tr}$, necessarily $u^* a \alpha_s(u) = a$. Taking $v = \pol(a)$ the polar part of $a$, we have found a nonzero partial isometry $v \in p\widetilde{M}p$ such that
\begin{equation}\label{specgap}
x v = v \alpha_s (x), \forall x \in P.
\end{equation}
Note that $vv^* \in P' \cap p\widetilde{M}p$ and $v^*v \in \alpha_s(P)' \cap p\widetilde{M}p$.

{\bf Step (2) : Proving $P \preceq_M L(\R)$ using the malleability of $(\alpha_t, \beta)$.} The rest of the proof, is very similar to the reasoning in \cite[Lemma 4.8, Theorem 6.1]{popamsri},  \cite[Theorem 4.1]{popamal1} and \cite[Theorem 4.3]{ipp} (see also \cite[Theorem 5.6]{houdayer3} and \cite[Theorem 3.4]{houdayer6}). For the sake of completeness, we will give a detailed proof.

By contradiction, assume $P \npreceq_M L(\R)$. The first task is to lift Equation $(\ref{specgap})$ to $s = 1$. Note that it is enough to find a nonzero partial isometry $w \in p\widetilde{M}p$ such that
\begin{equation*}
x w = w \alpha_{2s} (x), \forall x \in P.
\end{equation*}
Indeed, by induction we can go till $s = 1$ (because $s = 1/2^k$). Recall that $\beta(z) = z$, for every $z \in M$. Recall that $vv^* \in P' \cap p\widetilde{M}p$. Since $P \npreceq_{M} L(\R)$, we know from \cite[Theorem 2.4]{houdayer4} that $P' \cap p\widetilde{M}p \subset pMp$. In particular, $vv^* \in pMp$. Set $w = \alpha_s(\beta(v^*)v)$. Then, 
\begin{eqnarray*}
ww^* & = & \alpha_s(\beta(v^*) vv^* \beta(v)) \\
& = & \alpha_s(\beta(v^*) \beta(vv^*) \beta(v)) \\
& = & \alpha_s \beta(v^*v) \neq 0.
\end{eqnarray*}
Hence, $w$ is a nonzero partial isometry in $p\widetilde{M}p$. Moreover, for every $x \in P$,
\begin{eqnarray*}
w \alpha_{2s}(x) & = & \alpha_s(\beta(v^*) v \alpha_s(x)) \\
& = & \alpha_s(\beta(v^*) x v) \\
& = & \alpha_s(\beta(v^*x)v) \\
& = & \alpha_s(\beta(\alpha_s(x)v^*)v) \\
& = & \alpha_s\beta\alpha_s(x) \alpha_s(\beta(v^*)v) \\
& = & \beta(x) w \\
& = & xw.
\end{eqnarray*}

Since by induction, we can go till $s = 1$, we have found a nonzero partial isometry $v \in p\widetilde{M}p$ such that
\begin{equation}\label{inter}
xv = v\alpha_1(x), \forall x \in P.
\end{equation}
Note that $v^*v \in \alpha_1(P)' \cap pMp$. Moreover, since $\alpha_1 : p\widetilde{M}p \to p\widetilde{M}p$ is a $\ast$-automorphism, and $P \npreceq_{M} L(\R)$, \cite[Theorem 2.4]{houdayer4} gives
\begin{eqnarray*}
\alpha_1(P)' \cap p\widetilde{M}p & = & \alpha_1\left( P' \cap p\widetilde{M}p \right) \\
& \subset & \alpha_1(pMp).
\end{eqnarray*}
Hence $v^*v \in \alpha_1(pMp)$.

Since $P \npreceq_M L(\R)$, we know that there exists a sequence of unitaries $(u_k)$ in $P$ such that $\lim_{k} \|E_{L(\R)}(x^* u_k y)\|_{2, \Tr} \to 0$, for any $x, y \in M$. We need to go further and prove the following:

\begin{claim}\label{esperance}
$\forall a, b \in \widetilde{M}, \lim_k \| E_{M_2}(a^* u_k b) \|_{2, \Tr} = 0$.
\end{claim}

\begin{proof}[Proof of Claim $\ref{esperance}$]
Let $a, b \in (\widetilde{M})_1$ be either elements in $L(\R)$ or reduced words with letters alternating from $M_1 \ominus L(\R)$ and $M_2 \ominus L(\R)$. Write $b = y b'$ with
\begin{itemize}
\item $y = b$ if $b \in L(\R)$;
\item $y = 1$ if $b$ is a reduced word beginning with a letter from $M_2 \ominus L(\R)$;
\item $y =$ the first letter of $b$ coming from $M_1 \ominus L(\R)$ otherwise. 
\end{itemize}
Note that either $b' = 1$ or $b'$ is a reduced word beginning with a letter from $M_2 \ominus L(\R)$.  Likewise write $a = a' x$ with
\begin{itemize}
\item $x = a$ if $x \in L(\R)$;
\item $x = 1$ if $a$ is a reduced word ending with a letter from $M_2 \ominus L(\R)$;
\item $x =$ the last letter of $a$ coming from $M_1 \ominus L(\R)$ otherwise. 
\end{itemize}
Either $a' = 1$ or $a'$ is a reduced word ending with a letter from $M_2 \ominus L(\R)$. For any $z \in  M_1$, $xzy - E_{L(\R)}(xzy) \in M_1 \ominus L(\R)$, so that
\begin{equation*}
E_{M_2} (a z b) = E_{M_2}(a' E_{L(\R)}(x z y) b').
\end{equation*}
Since $\lim_k \|E_{L(\R)}(x u_k y)\|_{2, \Tr} = 0$, it follows that $\lim_k \|E_{M_2}(a u_k b)\|_{2, \Tr} = 0$ as well. Note that
\begin{equation*}
\mathcal{A}:= \mbox{span} \left\{ L(\R), (M_{i_1} \ominus L(\R)) \cdots (M_{i_n} \ominus L(\R)) : n \geq 1, i_1 \neq \cdots \neq i_n \right\}
\end{equation*}
is a unital $\ast$-strongly dense $\ast$-subalgebra of $\widetilde{M}$. What we have shown so far is that for any $a, b \in \mathcal{A}$, $\|E_{M_2}(a u_k b)\|_{2, \Tr} \to 0$, as $k \to \infty$. Let now $a, b \in (\widetilde{M})_1$. By Kaplansky density theorem, let $(a_i)$ and $(b_j)$ be sequences in $(\mathcal{A})_1$ such that $a_i \to a$ and $b_j \to b$ $\ast$-strongly. Recall that $(u_k)$ is a sequence in $P \subset p \widetilde{M} p$ with $\Tr(p) < \infty$. We have
\begin{eqnarray*}
\| E_{M_2}(a u_k b)\|_{2, \Tr} & \leq & \| E_{M_2}(a_i u_k b_j)\|_{2, \Tr} + \| E_{M_2}(a_i u_k (b - b_j)) \|_{2, \Tr} \\
& & + \| E_{M_2}((a - a_i) u_k b_j) \|_{2, \Tr} + \| E_{M_2}((a - a_i) u_k (b - b_j)) \|_{2, \Tr} \\
& \leq & \| E_{M_2}(a_i u_k b_j) \|_{2, \Tr} + \| a_i u_k p(b - b_j) \|_{2, \Tr} \\
& & + \| (a - a_i) pu_k b_j \|_{2, \Tr} + \| (a - a_i) u_k p(b - b_j) \|_{2, \Tr} \\
& \leq & \|E_{M_2}(a_i u_k b_j)\|_{2, \Tr} + 2 \|p (b - b_j) \|_{2, \Tr} + \| (a - a_i)p \|_{2, \Tr}  \\
\end{eqnarray*}
Fix $\varepsilon > 0$. Since $a_i \to a$ and $b_j \to b$ $\ast$-strongly, let $i_0, j_0$ large enough such that 
\begin{equation*}
2 \|p(b - b_{j_0}) \|_{2, \Tr} + \| (a - a_{i_0})p\|_{2, \Tr} \leq \varepsilon/2.
\end{equation*}
Now let $k_0 \in \N$ such that for any $k \geq k_0$,  
\begin{equation*}
\| E_{M_2}(a_{i_0} u_k b_{j_0})\|_{2, \Tr} \leq \varepsilon/2.
\end{equation*}
We finally get $\| E_{M_2}(a u_k b) \|_{2, \Tr} \leq \varepsilon$, for any $k \geq k_0$, which finishes the proof of the claim. 
\end{proof}

Recall that for any $x \in P$, $v^*xv = \alpha_1(x)v^*v$, by Equation $(\ref{inter})$. Moreover, $v^*v \in \alpha_1(pMp) \subset pM_2p$. So, for any $x \in P$, $v^*xv \in pM_2p$. Since $\alpha_1(u_k) \in \mathcal{U}(pM_2p)$, we get
\begin{eqnarray*}
\| v^*v \|_{2, \Tr}  & = & \| \alpha_1(u_k) v^*v \|_{2, \Tr} \\
& = & \| E_{M_2}(\alpha_1(u_k) v^*v) \|_{2, \Tr} \\
& = & \| E_{M_2}(v^* u_k v) \|_{2, \Tr} \to 0.
\end{eqnarray*}
Thus $v = 0$, which is a contradiction. 
\end{proof}

\begin{cor}\label{uniformcorollary}
Let $\mathcal{M} = \Gamma(H_\R, U_t)''$ and $M = \mathcal{M} \rtimes_\sigma \R$ be as above. Let $p \in L(\R) \subset M$ be a nonzero projection such that $\Tr(p) < \infty$. Let $P \subset pMp$ be a von Neumann subalgebra such that $P \npreceq_M L(\R)$. Then there exist $0 < \kappa < 1$, a sequence $(t_k)$ of positive reals and a sequence $(u_k)$ of unitaries in $\mathcal{U}(P)$ such that $\lim_k t_k = 0$ and $\|(E_M \circ \alpha_{t_k})(u_k)\|_{2, \Tr} \leq \kappa \|p\|_{2, \Tr}$, for every $k \in \N$. 
\end{cor}

\begin{proof}
Assume $P \npreceq_M L(\R)$. Using Theorem \ref{uniform}, we obtain that the deformation $(\alpha_t)$ does not converge uniformly on $\mathcal{U}(P)$. Combining this with Inequality $(\ref{trans})$ in Proposition \ref{transversality}, we get that there exist $0 < c < 1$, a sequence of positive reals $(t_k)$ and a sequence of unitaries $(u_k)$ in $\mathcal{U}(P)$ such that $\lim_{k} t_k = 0$ and $\| \alpha_{t_k}(u_k) - (E_M \circ \alpha_{t_k})(u_k) \|_{2, \Tr} \geq c\| p \|_{2, \Tr}$, $\forall k \in \N$. Since $\|\alpha_{t_k}(u_k)\|_{2, \Tr} = \| p \|_{2, \Tr}$, by Pythagora's theorem we obtain
\begin{equation*}
\|(E_M \circ \alpha_{t_k})(u_k)\|_{2, \Tr} \leq \kappa \| p \|_{2, \Tr}, \forall k \in \N.
\end{equation*}
where $\kappa = \sqrt{1 - c^2}$.
\end{proof}

\begin{rem}
Assume in Theorem $\ref{uniform}$ that the free Araki-Woods factor $\mathcal{M} = \Gamma(H_\R, U_t)''$ is a type ${\rm III_1}$ factor so that the core $M = \mathcal{M} \rtimes_\sigma \R$ is a type ${\rm II_\infty}$ factor. Then the $\Tr$-finite projection $p \in L(\R)$ can be replaced by {\em any} $\Tr$-finite projection in $M$. Indeed let $q \in M$ be a $\Tr$-finite projection. Since $M$ is a type ${\rm II_\infty}$ factor, $L(\R)$ is diffuse and $\Tr_{| L(\R)}$ is semifinite, we may find a projection $p \in L(\R)$ and a unitary $u \in \mathcal{U}(M)$ such that $u p u^* = q$.
\end{rem}

\section{Structural results: proofs of Theorems B and D}\label{theoBD}

\subsection{Weak containment of bimodules}

Let $M, N, P$ be any von Neumann algebras. For any $M, N$-bimodules $H, K$, denote by $\pi_H$ (resp. $\pi_K$) the associated $\ast$-representation of the algebraic tensor product $M \odot N^{\op}$ on $H$ (resp. on $K$). We say that $H$ is {\em weakly contained} in $K$ and denote it by $H \subset_{\weak} K$ if $\| \pi_H(T) \|_\infty \leq \| \pi_K(T) \|_\infty$, for every $T \in M \odot N^{\op}$. Recall that $H \subset_{\weak} K$ if and only if $H$ lies in the
closure (for the Fell topology) of all finite direct sums of copies of $K$. Let $H, K$ be $M, N$-bimodules. The following are true:

\begin{enumerate}
\item Assume that $H \subset_{\weak} K$. Then, for any $N, P$-bimodule $L$, we have $H \otimes_N L \subset_{\weak} K \otimes_N L$, as $M, P$-bimodules. Likewise, for any $P, M$-bimodule $L$, we have $L \otimes_M H \subset_{\weak} L \otimes_M K$, as $P, N$-bimodules (see \cite[Lemma 1.7]{anan95}).

\item A von Neumann algebra $B$ is amenable if and only if $L^2(B) \subset_{\weak} L^2(B) \otimes L^2(B)$, as $B, B$-bimodules.
\end{enumerate}

Let $B, M, N$ be von Neumann algebras such that $B$ is amenable. Let $H$ be any $M, B$-bimodule and let $K$ be any $B, N$-bimodule. Then, as $M, N$-bimodules, we have $H \otimes_B K \subset_{\weak} H \otimes K$ (straightforward consequence of $(1)$ and $(2)$).

We will be using from now on the notation introduced in Section $\ref{malleable}$. Let $\mathcal{M} = \Gamma(H_\R, U_t)''$ be a free Araki-Woods factor. Denote by $M = \mathcal{M} \rtimes_\sigma \R$ its continuous core.

\begin{lem}\label{weakcontainment}
Let $p \in L(\R)$ be a nonzero projection such that $\Tr(p) < \infty$. The $pM_1p, pM_1p$-bimodule $\mathcal{H} = L^2(p\widetilde{M}p) \ominus L^2(pM_1p)$ is weakly contained in the coarse bimodule $L^2(pM_1p) \otimes L^2(pM_1p)$.
\end{lem}

\begin{proof}
Set $B = L(\R)$. Let $p \in L(\R)$ be a nonzero projection such that $\Tr(p) < \infty$. By definition of the amalgamated free product $\widetilde{M} = M_1 \ast_{L(\R)} M_2$ (see \cite{voiculescu92} and \cite{ueda}), we have as $pM_1p, pM_1p$-bimodules
\begin{equation*}
L^2(p\widetilde{M}p) \ominus L^2(pM_1p) \cong \bigoplus_{n \geq 1} \mathcal{H}_n,
\end{equation*} 
where 
\begin{equation*}
\mathcal{H}_n = L^2(pM_1) \otimes_B \mathop{\overbrace{(L^2(M_2) \ominus L^2(B)) \otimes_B \cdots  \otimes_B (L^2(M_2) \ominus L^2(B))}}^{2n - 1} \otimes_B L^2(M_1p).
\end{equation*}
Since $B = L(\R)$ is amenable, the identity bimodule $L^2(B)$ is weakly contained in the coarse bimodule $L^2(B) \otimes L^2(B)$. From the standard properties of composition and weak containment of bimodules, it follows that as $pM_1p, pM_1p$-bimodules
\begin{equation*}
\mathcal{H}_n \subset_{\weak} L^2(pM_1) \otimes \mathop{\overbrace{(L^2(M_2) \ominus L^2(B)) \otimes \cdots  \otimes (L^2(M_2) \ominus L^2(B))}}^{2n - 1} \otimes L^2(M_1p).
\end{equation*}
Consequently, we obtain as $pM_1p, pM_1p$-bimodules
\begin{equation*}
\mathcal{H} = L^2(p\widetilde{M}p) \ominus L^2(pM_1p) \subset_{\weak} \bigoplus L^2(pM_1) \otimes L^2(M_1p).
\end{equation*}
Moreover, as a left $pM_1p$-module, $L^2(pM_1)$ is contained in $\bigoplus L^2(pM_1p)$. Likewise, the right $pM_1p$-module $L^2(M_1p)$ is contained in $\bigoplus L^2(pM_1p)$. Therefore, we get as $pM_1p, pM_1p$-bimodules
\begin{equation*}
\mathcal{H} = L^2(p\widetilde{M}p) \ominus L^2(pM_1p) \subset_{\weak} \bigoplus L^2(pM_1p) \otimes L^2(pM_1p).
\end{equation*}
\end{proof}

\subsection{The intermediate key result}

Let $\mathcal{M} = \Gamma(H_\R, U_t)''$ be a free Araki-Woods factor. Since $\mathcal{M}$ has the complete metric approximation property by Theorem A,  so do its core $M = \mathcal{M} \rtimes_\sigma \R$ and $pMp$, for any $\Tr$-finite nonzero projection $p \in M$ by Theorem $\ref{properties}$.

\begin{theo}\label{normalizer1}
Let $\mathcal{M} = \Gamma(H_\R, U_t)''$ be a free Araki-Woods factor. Denote by $\chi$ the corresponding free quasi-free state and by $M = \mathcal{M} \rtimes_{\sigma^\chi} \R$ the continuous core. Let $p \in L(\R)$ be a nonzero projection such that $\Tr(p) < \infty$. Let $P \subset pMp$ be an amenable von Neumann subalgebra. If $P \npreceq_M L(\R)$, then $\mathcal{N}_{pMp}(P)''$ is amenable.
\end{theo}

\begin{proof}
The proof is a generalization of the one of \cite[Theorem 3.5]{houdayer8} building on the work of Ozawa and Popa (see \cite[Theorem 4.9]{ozawapopa} and \cite[Theorem B]{ozawapopaII}). What is shown in \cite[Theorem 3.5]{houdayer8} is the following. Assume that $P \subset N$ are finite von Neumann algebras such that $P$ is amenable and $N$ has the c.m.a.p. Assume moreover that there are a finite von Neumann algebra $N \subset \widetilde{N}$ and trace-preserving $\ast$-homomorphisms $\alpha_t : N \to \widetilde{N}$ such that:
\begin{enumerate}
\item $\lim_{t \to 0} \|\alpha_t(x) - x\|_2 = 0$, for every $x \in N$.
\item There exists $0 < \kappa < 1$, a sequence of positive reals $(t_k)$ and a sequence of unitaries $(u_k)$ in $\mathcal{U}(P)$ such that $\lim_k t_k = 0$ and $\|(E_N \circ \alpha_{t_k})(u_k) \|_{2, \Tr} \leq \kappa \|p\|_{2, \Tr}$, for every $k \in \N$.
\item The $N, N$-bimodule $L^2(\widetilde{N}) \ominus L^2(N)$ is weakly contained in the coarse bimodule $L^2(N) \otimes L^2(N)$.
\end{enumerate}
Then $\mathcal{N}_N(P)''$ is amenable.

Now let $\mathcal{M} = \Gamma(H_\R, U_t)''$ be a free Araki-Woods factor. Denote by $\chi$ the corresponding free quasi-free state and by $M = \mathcal{M} \rtimes_{\sigma^\chi} \R$ the continuous core. Let $p \in L(\R)$ be a nonzero projection such that $\Tr(p) < \infty$. We know that $N = pMp$ has the c.m.a.p.\ since both $\mathcal{M}$ and $M$ have the c.m.a.p.\ (by Theorem A). Let $P \subset pMp$ be an amenable von Neumann subalgebra. The malleable deformation $(\alpha_t)$ clearly satisfies $(1)$. Since $P \npreceq_M L(\R)$, Corollary $\ref{uniformcorollary}$  yields $(2)$. Lemma $\ref{weakcontainment}$ finally yields $(3)$. Therefore $\mathcal{N}_{pMp}(P)''$ is amenable.
\end{proof}

\subsection{Proof of Theorem B}
Let $\mathcal{M}$ be a von Neumann algebra and let $\varphi, \psi$ be two faithful normal states on $\mathcal{M}$. Recall from Section $\ref{preliminaries}$ that through the natural $\ast$-isomorphism 
$$\Pi_{\varphi, \psi} : \mathcal{M} \rtimes_{\sigma^\varphi} \R \to
\mathcal{M} \rtimes_{\sigma^\psi} \R,$$ we will identify
$$(\pi_{\sigma^\varphi}(\mathcal{M}) \subset \mathcal{M}
\rtimes_{\sigma^\varphi} \R, \theta^\varphi, \Tr_\varphi) \mbox{ with
} (\pi_{\sigma^\psi}(\mathcal{M}) \subset \mathcal{M}
\rtimes_{\sigma^\psi} \R, \theta^\psi, \Tr_\psi),$$ and simply denote
it by $(\mathcal{M} \subset M, \theta, \Tr)$, where $\theta$ is the
dual action of $\R$ on the core $M$ and $\Tr$ is the semifinite
faithful normal trace on $M$ such that $\Tr \circ \theta_s =
e^{-s}\Tr$, for any $s \in \R$.

However, we need to pay attention to the following: whereas the
inclusion $\mathcal M \subset M$ does not depend on the state, there
are {\em a priori} two different copies of the abelian von Neumann
algebra $L(\R)$ inside $M$. To avoid any confusion, we will denote by
$\lambda^\varphi(s)$ (resp.\ $\lambda^\psi(s)$) the unitaries
implementing the modular action $\sigma^\varphi$ (resp.\ $\sigma^\psi$)
on $\mathcal{M}$. The following technical Proposition will be useful, as it
explains why we do not have to worry very much about the state. 

\begin{prop}\label{diffuse}
Let $\mathcal{M}$ be a von Neumann algebra. Let $A \subset \mathcal{M}$ be a
separable diffuse von Neumann subalgebra. Then, for any nonzero
projection $p \in A' \cap M$ with $\Tr(p) < \infty$, and any faithful
normal state $\varphi$ on $\mathcal{M}$, we have
\begin{equation*}
A p \npreceq_M \lambda^\varphi(\R)''.
\end{equation*}
\end{prop}

\begin{proof}
Fix $\varphi$ a faithful normal state on $\mathcal{M}$ and $p$ a nonzero $\Tr$-finite projection in $M$. Since $A \subset \mathcal{M}$ is diffuse and separable, any maximal abelian $\ast$-subalgebra in
$A$ is separable and diffuse, and thus isomorphic to $L^\infty([0,1])$. Therefore
there exists a sequence of unitaries $(u_n)$ in $A$ such that $u_n \to
0$ weakly. Observe that $Ap \subset pMp$ is a von Neumann subalgebra and that $(u_n p)$ are unitaries in $A p$.

Let $(q_m)$ be an increasing sequence of projections in $\lambda^\varphi(\R)''$ such that $q_m \to 1$ strongly and $\Tr(q_m) < \infty$. Let $x, y \in (M)_1$ and $\varepsilon > 0$. Since $\Tr(p) < \infty$, choose $m \in \N$ large enough such that 
\begin{equation*}
\|q_m x^* p - x^* p\|_{2, \Tr} + \|p y q_m - p y\|_{2, \Tr} < \varepsilon.
\end{equation*}
Observe now that the unital $\ast$-algebra
\begin{equation*}
\mathcal{E} := \left\{\sum_{s \in S} x_s \lambda^\varphi(s) : S \subset \R \mbox{ finite}, x_s \in \mathcal{M}\right\}
\end{equation*} 
is $\ast$-strongly dense in $M$, so that one can find nets $(x_i)_{i \in I}$ and $(y_j)_{j \in J}$ in $(\mathcal{E})_1$ such that $x_i \to px$ and $y_j \to py$ $\ast$-strongly. Since now $\Tr(q_m) < \infty$, one can find $(i, j) \in I \times J$, such that 
\begin{equation*}
\|q_m x^* p - q_m x_i^*\|_{2, \Tr} + \|p y q_m - y_j q_m\|_{2, \Tr} < \varepsilon.
\end{equation*}
For simplicity of notation write $L(\R) := \lambda^\varphi(\R)''$. For every $n \in \N$, we get
\begin{eqnarray*}
\|E_{L(\R)}(x^*p u_n p y)\|_{2, \Tr} & \leq & \|E_{L(\R)}(q_m x^*p u_n p y q_m)\|_{2, \Tr} + \varepsilon \\
& \leq & \|E_{L(\R)}(q_m x_i^* u_n y_j q_m)\|_{2, \Tr} + 2\varepsilon.
\end{eqnarray*}
Since $x_i, y_j \in (\mathcal{E})_1$, write 
\begin{eqnarray*}
x_i & = & \sum_{s \in S} x_s \lambda^\varphi(s) \\
y_j & = & \sum_{t \in T} y_t \lambda^\varphi(t),
\end{eqnarray*}
where $S, T \subset \R$ are finite and $x_s, y_t \in \mathcal{M}$. Therefore
\begin{eqnarray*}
E_{L(\R)}(q_m x_i^* u_n y_j q_m) = \sum_{(s, t) \in S \times T} \varphi(x_s^\ast u_n y_t) \lambda^\varphi(t - s) q_m.
\end{eqnarray*}
Since $\varphi$ is a faithful normal state on $\mathcal{M}$, one may regard $A \subset \mathcal{M} \subset \mathbf{B}(L^2(\mathcal{M}, \varphi))$. Since $u_n \to 0$ weakly in $A$, there exists $n_0 \in \N$ large enough such that $\forall n \geq n_0$, $\forall (s, t) \in S \times T$, 
\begin{equation*}
|\varphi(x_s^* u_n y_t)| \leq \frac{\varepsilon}{\|q_m\|_{2, \Tr} (|S| \cdot |T| + 1)}.
\end{equation*}
We get, for every $n \geq n_0$,
\begin{equation*}
\| E_{L(\R)}(q_m x_i^* u_n y_j q_m) \|_{2, \Tr} \leq \varepsilon.
\end{equation*}
Therefore, we have for every $n \geq n_0$,
\begin{equation*}
\|E_{L(\R)}(x^* pu_np y)\|_{2, \Tr} \leq 3\varepsilon.
\end{equation*}
By $(2)$ of Lemma $\ref{intertwining}$, we get $A p \npreceq_M \lambda^\varphi(\R)''$. 
\end{proof}

We are now ready to prove Theorem B. We will denote by $\chi$ the
corresponding free quasi-free state on $\mathcal{M}$. We prove the
result by contradiction. Assume that there exists a diffuse
nonamenable von Neumann subalgebra $\mathcal{N} \subset \mathcal{M}$
together with $E : \mathcal{M} \to \mathcal{N}$ a faithful normal
conditional expectation such that $\mathcal{N}$ has a Cartan
subalgebra $A \subset \mathcal{N}$. Observe that $A$ is necessarily
diffuse. Denote by $F : \mathcal{N} \to A$ the faithful normal
conditional expectation. Choose a faithful normal trace $\tau$ on
$A$. Write $\psi = \tau \circ F \circ E$. Observe that $\psi$ is a
faithful normal state on $\mathcal{M}$ such that $\psi \circ E = \psi$
and $A \subset \mathcal{N}^\psi$. Set $M = \mathcal{M}
\rtimes_{\sigma^\psi} \R$ and $N = \mathcal{N} \rtimes_{\sigma^\psi}
\R$ and notice that $\lambda^\psi(\R)''\subset A' \cap M$. Observe that since
$\mathcal{N}$ is a nonamenable von Neumann algebra, its core $N$ is
nonamenable as well. Take a nonzero $\Tr$-finite projection $p \in
\lambda^\psi(\R)''$ large enough such that $pNp$
is nonamenable. Since $(A \overline{\otimes} \lambda^\psi(\R)'')(1
\otimes p) \subset pNp$ is regular and $pNp$ is nonamenable, Theorem
\ref{normalizer1} implies that $(A \overline{\otimes}
\lambda^\psi(\R)'')(1 \otimes p) \preceq_M \lambda^\chi(\R)''$ and
thus $A(1 \otimes p) \preceq_M \lambda^\chi(\R)''$. Since $A$ is
diffuse, this contradicts Proposition \ref{diffuse}.

\subsection{Proof of Theorem D}

Let $\mathcal{M} = \Gamma(H_\R, U_t)''$ be a free Araki-Woods
factor. As usual, denote by $M = \mathcal{M}Ê\rtimes_\sigma \R$ its
continuous core, where $\sigma$ is the modular group associated with
the free quasi-free state $\chi$. Let $p \in L(\R) :=
\lambda^\chi(\R)''$ be a nonzero projection such that $\Tr(p) <
\infty$.

{\bf (1)} By contradiction, assume that there exists a maximal abelian
$\ast$-subalgebra $A \subset pMp$ for which $\mathcal{N}_{pMp}(A)''$ is
not amenable. Write $p - z \in \mathcal{Z}(\mathcal{N}_{pMp}(A)'')$
for the maximal projection such that $\mathcal{N}_{pMp}(A)''(p - z)$
is amenable. Then $z \neq 0$ and $\mathcal{N}_{pMp}(A)''z$ has no
amenable direct summand. Notice that
\begin{equation*}
\mathcal{N}_{pMp}(A)''z \subset \mathcal{N}_{zMz}(Az)''.
\end{equation*}
Since this is a unital inclusion (with unit $z$),
$\mathcal{N}_{zMz}(Az)''$ has no amenable direct summand
either. Moreover, $Az \subset zMz$ is still maximal abelian. Since
$L(\R)$ is diffuse, $\Tr_{| L(\R)}$ is semifinite and $M$ is a type ${\rm II_\infty}$ factor, we can
find a projection $p_0 \in L(\R)$ such that $p_0 \leq p$ and a unitary
$u \in \mathcal{U}(M)$ such that $u z u^* = p_0$. Observe that $A_0 =
uAzu^* \subset p_0 Mp_0$ is maximal abelian and $\mathcal{N}_{p_0 M
p_0}(A_0)''$ has no amenable direct summand. Therefore, we may assume
without loss of generality that $p = p_0$, i.e.  $A \subset pMp$ is a
maximal abelian $\ast$-subalgebra for which $\mathcal{N}_{pMp}(A)''$
has no amenable direct summand.

Theorem $\ref{normalizer1}$ yields $A \preceq_M L(\R)$. Thus there
exists $n \geq 1$, a nonzero $\Tr$-finite projection $q \in L(\R)^n$,
a nonzero partial isometry $v \in \mathbf{M}_{1, n}(\C) \otimes pM$
and a unital $\ast$-homomorphism $\psi : A \to L(\R)^n$ such that $x v
= v \psi(x)$, $\forall x \in A$. Write $q = \psi(p)$, $q' =
v^*v$. Note that $vv^* \in A' \cap pMp = A$ and $q' \in \psi(A)' \cap
q M^n q$. It follows that $q' (\psi(A)' \cap q M^n q) q' =
(\psi(A)q')' \cap q' M^n q'$. Since by spatiality $\psi(A)q' = v^*Av$
is maximal abelian, we get $q' (\psi(A)' \cap q M^n q) q' = \psi(A) q'
= v^*A v$. Thus $\psi(A)' \cap q M^n q$ has a type ${\rm I}$ abelian direct
summand. Moreover,
$$q(\mathcal{M}^\chi \overline{\otimes} L(\R))^nq \subset q(L(\R)' \cap M)^nq \subset \psi(A)' \cap q M^n q.$$
Recall that one of the following situations holds:
\begin{enumerate}
\item [(a)] $(U_t)$ contains a trivial or periodic subrepresentation of dimension $2$. In that case, $L(\F_2) \subset \mathcal{M}^\chi$.
\item [(b)] $(U_t) = \R \oplus (V_t)$, where $(V_t)$ is weakly mixing. In that case, $\mathcal{M}^\chi = L(\Z)$.
\item [(c)] $(U_t)$ is weakly mixing and then $\mathcal{M}^\chi = \C$. 
\end{enumerate}
The subcase (a) cannot occur because otherwise $\psi(A)' \cap q M^n q$ would be of type ${\rm II}$. 

Assume now that (b) occurs. We have $(U_t) = \R \oplus (V_t)$ where $(V_t)$ is weakly mixing. Then we have 
$$\mathcal{M} = \Gamma(H_\R, U_t)'' \simeq \Gamma(K_\R, V_t)'' \ast L(\Z),$$
and \cite[Proposition 1]{ueda3} implies that $L(\Z)$ is maximal abelian in $\mathcal{M}$. Therefore $B = L(\Z) \overline{\otimes} L(\R)$ is maximal abelian in $M$. Since $A \preceq_M L(\R)$, we get $A \preceq_M B$. Since $A \subset pMp$ and $B \subset M$ are both maximal abelian,  Proposition $\ref{intertwining-masa}$ yields $n \geq 1$, a nonzero partial isometry $v \in pM$ such that $vv^* \in A$, $v^*v \in B$ and $v^* A v = Bv^*v$. By spatiality, we get
$$\Ad(v^*)\left( \mathcal{N}_{vv^*Mvv^*}(A vv^*)''\right) = \mathcal{N}_{v^*vMv^*v}(Bv^*v)''.$$
On the one hand, $\mathcal{N}_{vv^*Mvv^*}(A vv^*)'' = vv^* \mathcal{N}_{pMp}(A)'' vv^*$ is not amenable, since $\mathcal{N}_{pMp}(A)''$ has no amenable direct summand. On the other hand, since $L(\Z) = \mathcal{M}^\chi$ is diffuse, Proposition $\ref{diffuse}$ implies $B v^*v = (L(\Z) \overline{\otimes} L(\R))v^*v \npreceq_M L(\R)$. Theorem $\ref{normalizer1}$ implies that $\mathcal{N}_{v^*v M v^*v}(Bv^*v)''$ is amenable. We have reached a contradiction.

Assume at last that (c) occurs. Since $(U_t)$ is weakly mixing, it follows that $\mathcal{M}^{\chi} = \C$ and $L(\R)$ is maximal abelian in $M$ by Proposition $\ref{masa}$. Proposition $\ref{intertwining-masa}$ yields $n \geq 1$, a nonzero partial isometry $v \in pM$ such that $vv^* \in A$, $v^*v \in L(\R)$ and $v^* A v = L(\R)v^*v$. By spatiality, we get
$$\Ad(v^*)\left( \mathcal{N}_{vv^*Mvv^*}(A vv^*)''\right) = \mathcal{N}_{v^*vMv^*v}(L(\R)v^*v)''.$$
On the one hand, $\mathcal{N}_{vv^*Mvv^*}(A vv^*)'' = vv^* \mathcal{N}_{pMp}(A)'' vv^*$ is not amenable, since $\mathcal{N}_{pMp}(A)''$ has no amenable direct summand. On the other hand, since $(U_t)$ is weakly mixing, $L(\R)$ is singular in $M$, i.e.\ $\mathcal{N}_{M}(L(\R))'' = L(\R)$. Therefore $\mathcal{N}_{v^*vMv^*v}(L(\R)v^*v)'' = L(\R)v^*v$. We have reached again a contradiction.

{\bf (2-a) Assume that $(U_t)$ is strongly mixing.} Let $P \subset pMp$ be a unital diffuse  amenable von Neumann subalgebra. By contradiction, assume that $ \mathcal{N}_{pMp}(P)''$ is not amenable. With the same reasoning as before, we may assume that $\mathcal{N}_{pMp}(P)''$ has no amenable direct summand.

Theorem $\ref{normalizer1}$ yields $P \preceq_M L(\R)$. Thus there exist $n \geq 1$, a nonzero  $\Tr$-finite projection $q \in L(\R)^n$, a nonzero partial isometry $v \in \mathbf{M}_{1, n}(\C) \otimes pM$ and a unital $\ast$-homomorphism $\psi : P \to qL(\R)^nq$ such that $x v = v \psi(x)$, $\forall x \in P$. Note that $vv^* \in P' \cap pMp \subset \mathcal{N}_{pMp}(P)''$ and $v^*v \in \psi(P)' \cap q M^n q$. Since $\psi(P) \subset qL(\R)^n q$ is a unital diffuse von Neumann subalgebra and the action $\R \curvearrowright \mathcal{M}$ is strongly mixing (see \cite[Proposition 2.4]{houdayer6}), \cite[Theorem 3.7]{houdayer6} yields $\QN_{qM^nq}(\psi(P))'' \subset qL(\R)^nq$. Thus we may assume that $v^*v = q$. Let $u \in \mathcal{N}_{pMp}(P)$. We have
\begin{eqnarray*}
v^* u v \psi(P) & = & v^* u Pv \\
& = & v^* P u v \\
& = & \psi(P) v^* u v.
\end{eqnarray*}
Hence $v^* \mathcal{N}_{pMp}(P)'' v \subset \QN_{qM^nq}(\psi(P))'' \subset qL(\R)^nq$. But $$\Ad(v^*) : vv^* \mathcal{N}_{pMp}(P)'' vv^* \to  q L(\R)^n q$$
is a unital $\ast$-isomorphism. Since $\mathcal{N}_{pMp}(P)''$ has no amenable direct summand, $vv^* \mathcal{N}_{pMp}(P)'' vv^*$ is not amenable. This contradicts the fact that $q L(\R)^n q$ is amenable.

{\bf (2-b) Assume that $U_t = \R \oplus V_t$ where $(V_t)$ is strongly mixing.} Observe that we have $\Gamma(H_\R, U_t)'' = \Gamma(K_\R, V_t)'' \ast L(\Z)$. If we denote by $u$ a generating Haar unitary for $L(\Z)$ and by $\mathcal{Q}_\infty = \ast_{n \in \Z} u^n \Gamma(K_\R, V_t)'' u^{-n}$ the infinite free product, we may regard $\Gamma(H_\R, U_t)''$ as the crossed product
$$\Gamma(H_\R, U_t)'' = \mathcal{Q}_\infty \rtimes \Z$$
where the action $\Z \curvearrowright \mathcal{Q}_\infty$ is the {\em free Bernoulli shift}. Observe that the modular group $(\sigma^\chi_t)$ acts trivially on $L(\Z)$. Moreover, $(\sigma^\chi_t)$ acts diagonally on $\mathcal{Q}_\infty$ in the following sense. Denote by $\psi$ the free quasi-free state on $\Gamma(K_\R, V_t)''$. Let $y_1, \dots, y_k \in\Gamma(K_\R, V_t)'' \ominus \C$, $n_1 \neq \cdots \neq n_k$, $x_i = u^{n_i} y_i u^{-n_i}$ and write $x = x_1 \cdots x_k$ for the corresponding reduced word in $\mathcal{Q}_\infty$. Then we have
$$\sigma^\chi_t(x) = u^{n_1} \sigma^\psi_t(y_1) u^{-n_1} \cdots u^{n_k} \sigma_t^\psi(y_k) u^{-n_k}.$$
The core $M$ is therefore given by
$$M = \mathcal{Q}_\infty \rtimes (\Z \times \R).$$
Since $(V_t)$ is assumed to be strongly mixing, it is straightforward to check that the action $\Z \times \R \curvearrowright \mathcal{Q}_\infty$ is strongly mixing (see \cite[Proposition 2.4]{houdayer6}).

We are now ready to prove that $pMp$ is strongly solid. Assume by
contradiction that it is not. As we did before, let $P \subset pMp$ be a
unital diffuse amenable von Neumann subalgebra such that $
\mathcal{N}_{pMp}(P)''$ has no amenable direct summand. Theorem
$\ref{normalizer1}$ yields $P \preceq_M L(\R)$ and hence $P \preceq_M
L(\Z) \overline{\otimes} L(\R)$. Thus there exists $n \geq 1$, a
nonzero $\Tr$-finite projection $q \in (L(\Z) \overline{\otimes}
L(\R))^n$, a nonzero partial isometry $v \in \mathbf{M}_{1, n}(\C)
\otimes pM$ and a unital $\ast$-homomorphism $\psi : P \to q(L(\Z)
\overline{\otimes} L(\R))^nq$ such that $x v = v \psi(x)$, $\forall x
\in P$. Note that $vv^* \in P' \cap pMp \subset
\mathcal{N}_{pMp}(P)''$ and $v^*v \in \psi(P)' \cap q M^n q$. Since
$\psi(P) \subset q(L(\Z) \overline{\otimes} L(\R))^n q$ is a unital
diffuse von Neumann subalgebra and the action $\Z \times \R
\curvearrowright \mathcal{Q}_\infty$ is strongly mixing, \cite[Theorem
3.7]{houdayer6} yields $v^* \mathcal{N}_{pMp}(P)'' v \subset q(L(\Z)
\overline{\otimes} L(\R))^nq$. But $$\Ad(v^*) : vv^*
\mathcal{N}_{pMp}(P)'' vv^* \to q (L(\Z) \overline{\otimes} L(\R))^n
q$$ is a unital $\ast$-isomorphism. Since $\mathcal{N}_{pMp}(P)''$ has
no amenable direct summand, $vv^* \mathcal{N}_{pMp}(P)'' vv^*$ is not
amenable. This contradicts the fact that $q(L(\Z) \overline{\otimes}
L(\R))^n q$ is amenable.

\begin{rem}
If we do not assume that $A \subset pMp$ is maximal abelian, the
assertion $(1)$ in Theorem D fails to be true. Indeed assume that
$(U_t)$ is almost periodic. Denote $L(\R) = \lambda^\chi(\R)''$,
where $\chi$ is the free quasi-free state, which is assumed to be
almost periodic. It is straightforward to check that the groupoid
normalizer $\mathcal{G}\mathcal{N}_M(L(\R))$ generates $M$. Since $L(\R)$ is
abelian, we have $\mathcal{G}\mathcal{N}_M(L(\R))'' = \mathcal{N}_M(L(\R))''$
by \cite[Lemme 2.2]{fang-groupoid}, so that $\mathcal{N}_M(L(\R))'' =
M$. Let $p \in L(\R)$ be a nonzero $\Tr$-finite projection. Since
$L(\R)$ is abelian, we finally get
$$pMp = \mathcal{N}_{pMp}(L(\R)p)'',$$
that is, $L(\R)p$ is regular in $pMp$. 
\end{rem}

\subsection{Further structural results for free products}

A free malleable deformation for (amalgamated) free products of von
Neumann algebras was discovered in \cite{ipp}. Using ideas and
techniques of \cite{{houdayer4}, {houdayer7}, {ipp}, {ozawapopa}} and
of the present paper, we obtain the following indecomposability
results for free products of von Neumann algebras:

\begin{theo}
Let $(\mathcal{M}_i, \varphi_i)$ be a family of von Neumann algebras
endowed with faithful normal states. Denote by $(\mathcal{M}, \varphi)
= \ast_{i \in I} (\mathcal{M}_i, \varphi_i)$ their free product.
\begin{enumerate}
\item Assume that $\mathcal{M}$ has the complete metric approximation
property. Then either $\mathcal{M}$ is amenable or $\mathcal{M}$ has
no Cartan subalgebra.
\item Assume that each $\mathcal{M}_i$ is hyperfinite. Let
$\mathcal{N} \subset \mathcal{M}$ be a diffuse von Neumann subalgebra
for which there exists a faithful normal conditional expectation $E :
\mathcal{M} \to \mathcal{N}$. Then either $\mathcal{N}$ is hyperfinite
or $\mathcal{N}$ has no Cartan subalgebra.
\end{enumerate}
\end{theo}

Observe that in $(2)$, a free product of hyperfinite von Neumann
algebras automatically has the complete metric approximation property
by \cite{RicardXu}.

\end{document}